\documentclass[a4paper,11pt]{amsart}

\usepackage{amsmath,amssymb}
\usepackage{amsfonts}
\usepackage{graphicx}
\usepackage{algorithmic}
\usepackage{mathtools}
\usepackage[top=3.5cm, bottom=3.5cm, left=3cm, right=3cm]{geometry}
\usepackage{enumerate}
\usepackage{siunitx}
\usepackage{subcaption}
\usepackage{pgfplotstable}
\pgfplotsset{
compat = newest,
tick label style = {font = \tiny},
legend style = {font = \tiny},
xlabel style={yshift=+0.5ex},
ylabel style={yshift=-1.0ex}
}
\usepackage{hyperref}
\usepackage[capitalise]{cleveref}

\newtheorem{lemma}{Lemma}[section]
\newtheorem{proposition}{Proposition}[section]

\newtheorem{remark}{Remark}[section]
\newtheorem{algorithm}{Algorithm}[section]

\title[Approximation of harmonic map flows]{Projection-free approximation of flows \\
of harmonic maps with quadratic constraint accuracy
and variable step sizes}
\author[G. Akrivis]{Georgios Akrivis}
\address{Department of Computer Science and Engineering,
University of Ioannina, 451\,10 Ioannina, Greece
\& Institute of Applied and Computational Mathematics, FORTH, 700\,13 Heraklion, Greece}
\email{akrivis@cse.uoi.gr}
\author[S. Bartels]{S\"oren Bartels}
\address{Department of Applied Mathematics, University of Freiburg, 79104 Freiburg i.~Br., Germany}
\email{bartels@mathematik.uni-freiburg.de}
\author[M. Ruggeri]{Michele Ruggeri}
\address{Department of Mathematics, University of Bologna, 40126 Bologna, Italy}
\email{m.ruggeri@unibo.it}
\author[J. Wang]{Jilu Wang}
\address{School of Science, Harbin Institute of Technology, Shenzhen 518055, China}
\email{wangjilu@hit.edu.cn}

\date{\today}

\def\hu{\widehat{u}}
\def\veps{\varepsilon}
\def\DD{{\rm D}}

\def\uni{{\rm uni}}
\def\ener{\text{ener}}
\def\stop{\text{stop}}
\def\d{{\mathrm d}}
\def\R{{\mathbb R}}

\begin{document}

\begin{abstract}
We construct and analyze a projection-free linearly implicit method for the approximation
of flows of harmonic maps into spheres. The proposed method is unconditionally energy stable and,
under a sharp discrete regularity condition, achieves second-order accuracy with respect to the constraint violation. 
Furthermore, the method accommodates variable step sizes
to speed up the convergence to stationary points and
to improve the accuracy of the numerical solutions near singularities,
without affecting the unconditional energy stability and the constraint violation property. 
We illustrate the accuracy in approximating the unit-length constraint and the performance of the 
method through a series of numerical experiments, and compare it with the linearly
implicit Euler and two-step BDF methods.
\end{abstract}



\thanks{\emph{Acknowledgments.}
The first author gratefully acknowledges the support and warm hospitality 
provided by the Harbin Institute of Technology, Shenzhen, China.
The second author gratefully acknowledges
financial support by the German Research Foundation (DFG)
through the Research Unit FOR 3013 \emph{Vector- and tensor-valued surface PDEs}
(Project No.\ 417223351).
The third author is a member of the `Gruppo Nazionale per il Calcolo Scientifico' (GNCS) of the Italian `Istituto Nazionale di Alta
Matematica' (INdAM), and was partially supported by GNCS (CUP E53C23001670001) and by the European Cooperation in Science and Technology (COST)
through the COST Action POLYTOPO (CA23134).
The work of the fourth author was partially supported by the Shenzhen Science and Technology Program (Grant No.\ RCJC20231211090005019)
and the National Natural Science Foundation of China (Grant No. 12471371).
}

\maketitle

\section{Introduction}

Let $\varOmega\subset \R^d$ denote a bounded domain with
Lipschitz boundary $\partial\varOmega,$ and let 
$\varGamma_\DD \subset \partial\varOmega$ be the Dirichlet part
of the boundary, of positive surface measure. 
Harmonic maps (into the sphere) are the stationary points of the Dirichlet energy functional
\[I[u]=\frac12\int_\varOmega |\nabla u|^2\, \d x,\]
among all vector fields $u \in H^1(\varOmega; \R^\ell)$
satisfying the unit-length constraint $|u|=1$ a.e.\ in $\varOmega$,
subject to Dirichlet boundary conditions $u|_{\varGamma_\DD}=u_\DD$. 
Here, $H^1(\varOmega; \R^\ell)$ denotes
the Sobolev space 
consisting of vector fields $u:\varOmega\to \R^\ell$ in $L^2(\varOmega;\R^\ell)$
with square-integrable gradients,
while $u_\DD: \varGamma_\DD \to \R^{\ell}$ is a given function,
which is assumed to be equal to the trace of a function $\tilde u_\DD \in H^1(\varOmega; \R^\ell)$ with
$|\tilde u_\DD|=1$ a.e.\ in $\varOmega$.
The resulting Euler--Lagrange equations are
\begin{equation}\label{EL}
\begin{alignedat}{2}
&-\varDelta u = \lambda u, \quad 
&&|u| = 1 \ \text{ in } \varOmega, \\
&u = u_\DD \ \text{ on } \varGamma_\DD, \quad
&&\frac {\partial u}{\partial n} = 0 \ \text{ on } \partial\varOmega \setminus \varGamma_\DD.
\end{alignedat}
\end{equation}
The function $\lambda = |\nabla u|^2$ is the 
Lagrange multiplier related to the unit-length constraint;
see \cite[Ch.\ 7]{Bartels-book} and references therein. 

Gradient flows provide an attractive tool to solve the Euler--Lagrange equations as these decrease the Dirichlet 
energy along trajectories.  The simplest case corresponds to the $L^2$-gradient flow and it is known as the harmonic 
map heat flow (into spheres)
\begin{equation}\label{eq:heat flow}
\partial_tu-\varDelta u=|\nabla u|^2u,\quad |u|^2=1  \ \text{ in } \varOmega,  
\end{equation}
subject to initial and boundary conditions $u(0,\cdot)=u^0\in  {\mathcal A}_{u_\DD}$ and $u(t,\cdot)|_{\varGamma_\DD}=u_\DD$, 
where the class of  admissible vector fields ${\mathcal A}_{u_\DD}$ is defined by 
${\mathcal A}_{u_\DD}:=\{u\in H^1(\varOmega;  {\mathbb S}^{\ell-1}):  
\ u|_{\varGamma_\DD}=u_\DD\}$
and ${\mathbb S}^{\ell-1}:=\{v\in \R^\ell: |v|=1\}$ denotes  the unit sphere in $\R^\ell$. 
Under our assumptions, $\tilde u_\DD\in {\mathcal A}_{u_\DD};$ in particular, ${\mathcal A}_{u_\DD}\neq \emptyset.$

Equation \eqref{eq:heat flow} and various generalizations appear in numerous applications, including the Landau--Lifshitz--Gilbert (LLG) equation 
for magnetization dynamics \cite{AS}, models of nematic liquid crystals \cite{BFP}, and geometric evolution 
equations describing mean curvature flow of surfaces \cite{KLL1,KLL2}. A common feature in these applications
is that the exact solution inherently satisfies a unit-length constraint. 
This intrinsic requirement has inspired the development of efficient numerical methods that preserve 
the unit-length property at the discrete level. 

One straightforward approach involves the renormalization of numerical solutions using post-processing 
techniques to restore the crucial unit-length attribute to the solutions. 
Numerical methods that incorporate a projection step to satisfy the constraint at certain nodes exactly
lead to restrictions concerning the step size or the class of admissible triangulations; see \cite{Bartels-2005,BBFP,Al2,AKST}. 
Newton schemes can in general only be guaranteed to converge locally in the neighborhood of a sufficiently regular energy-stable solution~\cite{BPW}.
Methods that satisfy the constraint everywhere make use of appropriate nonlinear interpolation procedures and are known
to be optimally convergent for certain sufficiently regular solutions; see \cite{GHS}. The optimal convergence of piecewise 
affine finite element discretizations with nodal constraints has been established in \cite{HTW,BPW}.
An alternative approach to develop higher-order schemes for the
harmonic map heat flow and the LLG equation can be based on equivalent
unconstrained formulations; cf.~\cite{BaLuPr}. This however requires
the solution of nonlinear systems of equations in each time step.

The standard variational formulation of the Euler--Lagrange equations \eqref{EL}
is to seek a vector field $u\in {\mathcal A}_{u_\DD}$ such that
\begin{equation}\label{eq:varEL1}
(\nabla u, \nabla v)=\int_{\varOmega} |\nabla u|^2u\cdot v \, \d x 
 \end{equation}
for all $v\in H_\DD^1(\varOmega; \R^\ell)\cap L^\infty(\varOmega; \R^\ell)$; 
the test space in \eqref{eq:varEL1} is dictated by the fact that $|\nabla u|^2u\in L^1(\varOmega; \R^\ell)$.
Here, $H^1_\DD(\varOmega;\R^\ell)$ denotes the subspace of $H^1(\varOmega;\R^\ell)$
made of all functions with vanishing traces on $\varGamma_\DD$.
An alternative formulation, due to Alouges \cite{Al}, is
\begin{equation}\label{eq:varEL2}
(\nabla u, \nabla v)=0  
 \end{equation}
for all $v\in T_u$, where
\[T_u:=\{v\in H_\DD^1(\varOmega; \R^\ell):  v\cdot u =0 \text{ a.e.\ in } \varOmega\}\]
represents a solution-dependent test space, a closed subspace of $H_\DD^1(\varOmega; \R^\ell)$; see, e.g., \cite[Ch.~7]{Bartels-book}
for further details. The alternative form  \eqref{eq:varEL2} appears much simpler. 
Even though the test space  $T_u$ depends on the solution $u$, the primary advantage of~\eqref{eq:varEL2}
is that it can be easily linearized; see~\cite{Al}.
Based on formulation \eqref{eq:varEL2}, several projection-free linear schemes are proposed in \cite{Bartels-2016}; 
the resulting constraint violation is controlled linearly by the step size, independently of the number of iterations.
The approach has then been used for various related problems in~\cite{AHPPRS,HPPRSS,NRY}.

We now describe the main idea of the projection-free linearly implicit Euler scheme.
In \cite{Bartels-2016}, semi-implicit time discretization schemes are introduced for gradient flow problems represented by 
\begin{equation}\label{eq:grad-flow}
(u_t,v)_\star + (\nabla u,\nabla v) = 0,
 \end{equation}
subject to initial and boundary conditions $u(\cdot,0)=u^0\in {\mathcal A}_{u_\DD}$ and $u(\cdot,t)|_{\varGamma_\DD}=u_\DD$. 
The variational formulation~\eqref{eq:grad-flow} holds for a.e.\ $t \in [0,T]$ for some final time $T>0$,
and for all test functions $v\in T_{u(t)}$.
The pointwise unit-length constraint is imposed in the equivalent form $u_t\cdot u = 0$.
The inner product $(\cdot,\cdot)_\star$ can be either the $L^2$- or the $H^1$-inner product,
denoted by $(\cdot,\cdot)$ and $(\nabla \cdot,\nabla \cdot)$, respectively. 
Equation~\eqref{eq:grad-flow} then represents either the $L^2$- or the $H^1$-gradient
flow for harmonic maps into spheres, respectively.

Starting from $u^0$, the projection-free linearly implicit Euler 
method iteratively computes a sequence $(u^n)_{n \geqslant 1}$ 
of approximations as follows:
For a fixed step size $\tau,$ let $d_t u^n := (u^n-u^{n-1})/\tau$ denote 
the backward difference quotient.
At each iteration, one seeks $d_t u^n\in H_\DD^1(\varOmega; \R^\ell)$, 
satisfying the discrete linearized unit-length condition $d_t u^n \cdot u^{n-1} = 0$, such that 
\begin{equation}\label{eq:iEa}
(d_t u^n,v)_\star + (\nabla u^{n-1}+\tau \nabla d_t u^n,\nabla v) = 0
 \end{equation}
for all $v\in H_\DD^1(\varOmega; \R^\ell)$ with $v\cdot u^{n-1}=0.$ 
The new approximation is then given by 
$u^n:=u^{n-1}+\tau  d_t u^n$
and satisfies the required Dirichlet boundary condition $u^n|_{\varGamma_\DD}=u_\DD.$
Note that \eqref{eq:iEa} can be reformulated as 
\begin{equation}\label{eq:iEb}
(d_t u^n,v)_\star + (\nabla u^n,\nabla v) = 0
 \end{equation}
for all $v\in H_\DD^1(\varOmega; \R^\ell) $ with $ v\cdot u^{n-1}=0.$ 
The iteration is unconditionally well posed and energy decreasing; 
indeed, choosing $v = \tau \, d_t u^n$ in \eqref{eq:iEb} yields 
\begin{equation}\label{eq:iE-energ-stab}
\frac12 \|\nabla u^n\|^2 - \frac12 \|\nabla u^{n-1}\|^2 + \tau \|d_t u^n\|_\star^2 + \frac{\tau^2}{2} \|\nabla d_t u^n\|^2 = 0,
 \end{equation}
which implies, for every $m \geqslant 1$, the energy identity
\begin{equation}\label{eq:iE-energ-stab2}
\frac12 \|\nabla u^m\|^2
+ \tau \sum_{n=1}^m \|d_t u^n\|_\star^2
+ \frac{\tau^2}{2} \sum_{n=1}^m \|\nabla d_t u^n\|^2
= \frac12 \|\nabla u^0\|^2.
 \end{equation}
This yields the summability of the discrete time derivatives $\|d_t u^n\|_\star^2$ and
hence the weak convergence of subsequences to solutions of~\eqref{eq:varEL2}. 
A bound for the constraint violation thus follows from the orthogonality
condition and $|u^0| =1$, as we have
\begin{equation}\label{eq:iE-constr-viol} 
|u^m|^2 -1 = |u^{m-1}|^2 -1 + \tau^2 |d_t u^m|^2 = \dotsb
= \tau^2 \sum_{n=1}^m |d_t u^n|^2.
\end{equation}
By taking the $L^1$-norm of this identity, the sum on the right-hand side is bounded by 
$\tau (c_\star/2) \|\nabla u^0\|^2$ provided that the induced norm $\|\cdot \|_\star$ controls 
the $L^2$-norm up to a factor $c_\star^{1/2}$ (see~\eqref{eq:metric_equivalence} below).

The iterative scheme \eqref{eq:iEa} is built upon the weak formulation \eqref{eq:grad-flow} by Alouges 
and an implicit Euler temporal discretization. 
Recently, linearly implicit backward difference formula (BDF) schemes were introduced in \cite{AFKL} 
for the LLG equation.  For the harmonic map heat flow, a numerical approximation 
employing a nodal treatment of the unit-length constraint was proposed in \cite{BKW}. 
More recently, An, Gao, and Sun \cite{AGS} designed semi-implicit Euler and Crank--Nicolson 
finite difference projection methods for the LLG equation.
A similar semi-implicit approach, but based on the two-step BDF method, was proposed in~\cite{CWX}.
These works have yielded optimal/quasi-optimal error estimates, leading to 
bounds on the constraint violation, in situations involving sufficiently regular solutions. 

In this paper, we design a projection-free linearly implicit $(\theta,\mu)$-method for flows of harmonic maps, 
utilizing Alouges' weak formulation, and address its approximation properties when the smoothness of the solution is not guaranteed. 
Here, $0 < \theta \leqslant 1$ and $0 \leqslant \mu \leqslant 1$ denote two parameters of the method
characterizing its energy dissipation and its approach to realize the unit-length constraint,
respectively; see \cref{SSe:2.3} below.
The proposed method offers several key advantages: 
(1) It is unconditionally energy stable and achieves second-order accuracy in approximating the unit-length constraint, 
requiring only minimal regularity conditions on the solution, akin to the projection-free two-step BDF method proposed 
in \cite{ABP}; see \cref{pro:energy,prp:Contraint violation2} below.
(2) It accommodates variable step sizes, which paves the way to the application of acceleration techniques
for achieving faster convergence to stationary points and adaptive approaches
to improve the accuracy of numerical solutions near singularities
(such as a singularity at $t = 0$ for nonsmooth initial data or a finite blow-up for smooth initial data),
without affecting the unconditional energy stability and the constraint violation property.
This capability represents a notable advantage over the projection-free two-step BDF
method proposed in \cite{ABP}, as stability for BDF methods in the case of variable step sizes is a delicate matter; see, e.g., \cite{ACHYZ}.
We illustrate the accuracy in approximating the unit-length constraint and the 
performance of the proposed method through a series of numerical experiments,
comparing it with the linearly implicit Euler and two-step BDF methods.

To fix the ideas, we develop our theory for the problem of approximating flows of harmonic maps (into the sphere),
which serves as a prototype of a geometrically constrained partial differential equation.
Using the same ideas, projection-free linearly implicit iterative schemes
with second-order accuracy in the constraint approximation
supporting variable step sizes can be designed and analyzed for a broad class of problems,
e.g., in micromagnetics, liquid crystal theory, and bending theory.
We stress that our results are developed for a semi-discrete method,
in which only the time discretization of~\eqref{eq:grad-flow} is considered,
but hold verbatim if a spatial discretization with a nodal treatment of the constraint is used;
see~\cite{Bartels-2016},
where a fully discrete method based on first-order finite elements
for the linearly implicit Euler method is introduced and analyzed.
Moreover,
the method can be extended to more general target manifolds than the unit sphere,
e.g., to manifolds that can be characterized as the zero level set of a $C^2$-function
satisfying certain growth conditions.
Again, we refer to~\cite{Bartels-2016},
where this extension is developed for the method based on the linearly implicit Euler method.

The paper is organized as follows. 
In \cref{Se:2}, we introduce the projection-free linearly implicit $(\theta,\mu)$-method and the main theoretical results. 
For the sake of comparison, we also discuss the implicit Euler method and the two-step BDF method. 
The extension of the proposed method to accommodate variable step sizes and the corresponding analyses 
of energy decay and constraint violation are presented in \cref{Se:3}. 
Numerical results are provided in \cref{Se:num} to support the theoretical analysis and illustrate the performance 
of the schemes. 

In this work, we use standard notation for differential operators and Lebesgue and Sobolev spaces. 
We let $|\cdot|$ denote the Euclidean length of vectors as well as the Frobenius norm of matrices, 
and $\|\cdot\|$ the $L^2$-norm of functions or vector fields.

\section{Projection-free methods and main results}\label{Se:2}
 
In this section, we introduce a projection-free linearly implicit method for
the time discretization of flows of harmonic maps, 
which is proved to be of quadratic constraint accuracy.
To evaluate and compare the effectiveness of the proposed scheme, we also discuss 
the projection-free two-step BDF method recently introduced and analyzed in \cite{ABP}. 

\subsection{The linearly implicit $\boldsymbol{(\theta,\mu)}$-method}\label{SSe:2.3}

The implicit Euler method~\eqref{eq:iEa} proposed in \cite{Bartels-2016} is known as the simplest projection-free method for harmonic maps, 
and exhibits only first-order convergence
with respect to the step size of the error in the unit-length constraint.
In this subsection, we generalize the approach to obtain a linearly implicit $(\theta,\mu)$-method.

Let $t_n=n\tau$ for all $n \geqslant 0$, where $\tau>0$ denotes a fixed step size.
For $n \geqslant 1$ and a sequence $(u^n)_{n \geqslant 0}$ in a Hilbert space, we define 
\begin{equation*}
d_tu^n
:=
\frac{u^n-u^{n-1}}{\tau},
\quad
u^{n-1+\theta}
:=
u^{n-1}+\theta\tau d_tu^{n},
\quad
\widehat u^{n-1+\mu}
:=
u^{n-1}+\mu\tau d_tu^{n-1},
\end{equation*}
where $0\leqslant \theta,\mu\leqslant 1$.
Here,
$u^{n-1+\theta}$ denotes the interpolated value at $t_{n-1+\theta}=t_{n-1}+\theta\tau$ 
of the linear interpolant based on $(t_{n-1},u^{n-1})$ and $(t_{n},u^{n})$
(serving as an approximation to $u(t_{n-1+\theta})$),
while
$\widehat u^{n-1+\mu}$ denotes the extrapolated value at $t_{n-1+\mu}=t_{n-1}+\mu\tau$ 
of the linear interpolant based on $(t_{n-2},u^{n-2})$ and $(t_{n-1},u^{n-1})$
(serving as an approximation to $u(t_{n-1+\mu})$).
Note that, even if $\theta = \mu$,
$u^{n-1+\theta} \neq \widehat u^{n-1+\mu}$ in general.

Since ${\widehat u}^{n-1+\mu}$ relies on two previous approximations for positive $\mu$,
for the given initial value $u^0,$ we first determine a second starting approximation $u^1$ by employing
a single step of the linearly implicit Euler method \eqref{eq:iEa}. 
Subsequently, for given approximations  $u^{n-2}$ and $u^{n-1},$ 
we seek $d_t u^n\in H_\DD^1(\varOmega; \R^\ell)$, 
satisfying the linearized unit-length condition $d_t u^n \cdot {\widehat u}^{n-1+\mu} = 0,$
such that 
\begin{equation}\label{eq:CNa}
(d_t u^n, v)_\star
+  (\nabla u^{n-1+\theta},\nabla v)
=
(d_t u^n, v)_\star
+  (\nabla \big [u^{n-1} +\theta  \tau d_t u^n\big ],\nabla v)
= 0
\end{equation}
for all $v\in H_\DD^1(\varOmega; \R^\ell)$ with $v\cdot{\widehat u}^{n-1+\mu}=0$. 
Thus, upon computing $d_t u^n,$ a new approximation is defined as
\[u^n:=u^{n-1} +  \tau d_t u^n.\]
The new approximation $u^n$ satisfies the required Dirichlet boundary condition $u^n|_{\varGamma_\DD}=u_\DD,$
provided the approximation $u^{n-1}$ at the previous time level satisfies this condition. 

We will demonstrate that, for suitable choices of the parameters $\theta$ and $\mu$,
the linearly implicit $(\theta,\mu)$-method \eqref{eq:CNa} is energy decreasing
and, under a sharp regularity condition on the numerical solution,
satisfies a constraint violation estimate of second order.
We present the main theoretical results in the next subsection.

\begin{remark}[relevant choices for $\theta$ and $\mu$]
We now comment on the choice of the parameters $\theta$ and $\mu$ in the proposed method.
The resulting schemes will be compared with each other numerically in \cref{Se:num}.\\
(i) For $\theta=1$ and $\mu=0$, \eqref{eq:CNa} reduces to the implicit Euler method \eqref{eq:iEa}. \\
(ii) For $\theta=1/2$ and $\mu=1/2$, we get the linearly implicit midpoint method,
which will be unconditionally energy stable and will achieve second-order accuracy in approximating the unit-length constraint; 
see \cref{pro:energy,prp:Contraint violation2} below.\\
(iii) For $\theta=1$ and $\mu=1/2$, we get a modified implicit Euler method. 
In this case, the variational formulation \eqref{eq:CNa} to be solved at each iteration is the same as that of the linearly implicit 
Euler method \eqref{eq:iEa}. However, like in the midpoint method for $n \geqslant 2$, the orthogonality constraint  is considered 
with respect to the extrapolated value ${\widehat u}^{n-1/2}$.
We will see that the modified linearly implicit Euler method will be characterized by the same energy decay property 
\eqref{eq:iE-energ-stab} as the standard linearly implicit Euler method,
but will have the same quadratic constraint accuracy as the midpoint method;
see \cref{pro:energy,prp:Contraint violation2} below. 
\end{remark}

We now summarize the proposed $(\theta,\mu)$-method \eqref{eq:CNa} in the following algorithm. 
 
\begin{algorithm}[$(\theta,\mu)$-method]
\label{alg}
\begin{algorithmic}
\STATE{Choose $u^0\in H^1(\varOmega;\R^\ell)$ with $u^0|_{\varGamma_\DD} = u_\DD$ and $|u^0| = 1$.}
\STATE{(0) Compute $d_t u^1 \in H^1_\DD(\varOmega;\R^\ell)$ such that $d_t u^1 \cdot u^0 = 0$ and
\[
(d_t u^1, v)_\star + (\nabla [u^0+\tau d_t u^1],\nabla v) = 0
\]
for all $v\in H^1_\DD(\varOmega;\R^\ell)$ with $v\cdot u^0 = 0$; set $u^1 = u^0+\tau d_t u^1$ and $n=2$.}
\STATE{(1) Set ${\widehat u}^{n-1+\mu}:= u^{n-1}+\mu\tau\d_tu^{n-1} $.
Compute $d_t u^n \in H^1_\DD(\varOmega;\R^\ell)$ such that $d_t u^n \cdot {\widehat u}^{n-1+\mu} = 0$ and
\[
(d_t u^n, v)_\star + (\nabla [u^{n-1}+\theta\tau d_t u^n],\nabla v) = 0
\]
for all $v\in H^1_\DD(\varOmega;\R^\ell)$ with $v\cdot {\widehat u}^{n-1+\mu} = 0$.
Then, set $u^n = u^{n-1}+ \tau d_t u^n$.}
\STATE{(2) Stop if $\|d_t u^n\|_\star + \theta\tau \|\nabla (d_t u^n)\|  \leqslant \veps_{\rm stop}$ or if $n \tau \geqslant T$.}
\STATE{(3) Increase $n \to n+1$ and continue with~(1).}
\end{algorithmic}
\end{algorithm}
 
\begin{remark}
The first stopping criterion used in \cref{alg}, i.e., step~(2),
applies to the case of energy minimization
and is chosen in such a way that, if we define the approximate harmonic map
generated by the algorithm as
$u_h^{N_\stop-1}$,
where $N_\stop$ is the smallest integer for which the stopping criterion is met,
then the following relationship holds
\begin{equation}\label{stop 1}
(\nabla u_h^{N_\stop-1},\nabla v) = R(v)
\text{ for all admissible } v
\quad
\text{and}
\quad
\| R \|_{H^1(\varOmega)^\star} \leqslant \veps_{\text{stop}} . 
\end{equation}
The aim is to ensure a fair comparison of the methods in numerical experiments:
By utilizing the same tolerance for the standard Euler, modified Euler, and midpoint methods, 
the resulting approximate harmonic maps
satisfy the harmonic map equation~\eqref{eq:varEL2} with a comparable accuracy
(in the sense that the dual norm of the residual can be bounded by the same tolerance).
The second criterion applies to the case in which one is interested in 
approximating the gradient flow dynamics until a certain prescribed final time $T>0$.
\end{remark}

\subsection{Energy decay and constraint violation for the $\boldsymbol{(\theta,\mu)}$-method} 
 
Hereafter,
we assume that the norm induced by the inner product $(\cdot, \cdot)_\star$
satisfies
\begin{equation} \label{eq:metric_equivalence}
\| v \| \leqslant c_\star^{1/2} \| v \|_\star
\quad
\text{for all }
v \in H^1_\DD(\varOmega;\R^\ell).
\end{equation}
This inequality trivially holds for both the $L^2$- and $H^1$-norms.
 
In the following proposition, we show the unconditional stability of the method.

\begin{proposition}[energy identity]\label{pro:energy}
The sequence generated by \cref{alg} satisfies, for all $m\geqslant 2$,
the energy identity
\begin{equation}\label{eq:CV-energ-stab}
\frac{1}{2}\|\nabla u^m\|^2
+\tau \sum_{n=1}^m \|d_t u^n\|_\star^2 
+ \frac{\tau^2}{2} \|\nabla d_tu^1\|^2
+ \Big(\theta-\frac12\Big) \tau^2\sum_{n=2}^m\|\nabla d_tu^n\|^2
= \frac12 \|\nabla u^0\|^2.
 \end{equation}
\end{proposition}

\begin{proof}
Testing \eqref{eq:CNa} by the admissible test function $v:=d_t u^n,$ 
we have 
\[\|d_t u^n\|_\star^2 +(\nabla [u^{n-1}+\theta\tau d_tu^n],\nabla  d_t u^n) = 0,\]
i.e.,
\[\|d_t u^n\|_\star^2 + (\nabla u^{n-1/2}+(\theta-1/2) \tau \nabla d_t u^n,\nabla  d_t u^n ) = 0,\]
and thus
\[\|d_t u^n\|_\star^2 +\frac 1{2\tau} \big (\|\nabla u^n\|^2-\|\nabla u^{n-1}\|^2\big ) 
+\Big(\theta-\frac12\Big)\tau\| \nabla d_tu^n\|^2
=0.\]
Multiplying this relation by $\tau$ and summing over $n$ from $n=2$ to $n=m$
yields
\begin{equation*}
\frac{1}{2}\|\nabla u^m\|^2+\tau  \sum_{n=2}^m \|d_t u^n\|_\star^2 
+\Big(\theta-\frac12\Big) \tau^2\sum_{n=2}^m\|\nabla d_tu^n\|^2
=\frac12 \|\nabla u^1\|^2.
\end{equation*}
Using the energy identity~\eqref{eq:iE-energ-stab}
for the first step ($n=1$) performed with the linearly implicit Euler method, we
obtain the asserted identity.
\end{proof}
 
We now proceed with studying the constraint violation properties, which provide an unconditional linear rate 
and a quadratic rate under a mild but necessary discrete regularity condition.

For a sequence $(v^n)_{n \geqslant 0}$,
let $d_t^2 v^n$ denote the second difference quotient
\[d_t^2 v^n := \frac{1}{\tau} (d_t v^n - d_tv^{n-1})
=\frac{1}{\tau^2} (v^n - 2 v^{n-1} + v^{n-2}),\quad n\geqslant 2.\] 
The following lemma will be useful in the analysis; it presents a discrete version of the identity $\partial_t |v|^2 = 2 \partial_t v\cdot v$. 

\begin{lemma}[discrete chain rule]\label{Le:orthog-time-deriv}
For a sequence $(v^n)_{n \geqslant 0}$, we have
\[2 d_t v^n \cdot {\widehat v}^{n-1+\mu}
=d_t|v^n|^2- \mu \tau^2 d_t|d_t v^n|^2 - \mu \tau^3 |d_t^2 v^n|^2
-\left(1-2\mu\right)\tau|d_tv^n|^2 ,
\]
for $n \geqslant 2$. 
\end{lemma}

\begin{proof}
We start by splitting the left-hand side of the asserted identity as
\begin{align*}
2 d_t v^n \cdot {\widehat v}^{n-1+\mu}
&= 2 d_t v^n \cdot v^{n-1/2}
- 2d_t v^n \cdot (v^{n-1/2}-{\widehat v}^{n-1+\mu}) \\
&=d_t|v^n|^2 - 2\mu\tau d_t v^n \cdot (d_t v^n-d_tv^{n-1}) - \left(1-2\mu \right)\tau|d_tv^n|^2 ,
\end{align*}
where we have used the relation
\begin{align*}
v^{n-1/2}-{\widehat v}^{n-1+\mu}
&=
v^{n-1/2}-v^{n-1}-\mu\tau d_tv^{n-1} \\
&=
\tau d_tv^n/2 - \mu\tau d_tv^{n-1} \\
&=
\mu\tau(d_tv^n-d_tv^{n-1}) +\left(1/2-\mu\right)\tau d_tv^n .
\end{align*} 
Note that 
\begin{equation*}
\begin{split}
 d_t v^n \cdot (d_tv^n-d_tv^{n-1})
&=\big (|d_t v^n|^2-|d_t v^{n-1}|^2\big)/2
+ \tau^2|d_t^2 v^n|^2/2 \\
&=\tau d_t|d_tv^n|^2/2
+ \tau^2|d_t^2 v^n|^2/2;
\end{split}
\end{equation*}
then, the desired result follows immediately. 
\end{proof}

From \cref{Le:orthog-time-deriv} we deduce that,
if $d_t u^n \cdot {\widehat u}^{n-1+\mu} = 0$, then
\begin{equation}\label{eq:orthog-time-deriv2}
d_t|u^n|^2=\mu \tau^2 d_t|d_t u^n|^2+ \mu \tau^3 |d_t^2 u^n|^2
+ \left(1-2\mu\right)\tau|d_tu^n|^2 ,
\end{equation}
for all $n\geqslant 2$.
Using this identity, in the following proposition, we establish the constraint violation property of the method.

\begin{proposition}[constraint violation error]\label{prp:Contraint violation}
The sequence generated by \cref{alg} satisfies, for all $m\geqslant 2$,
\begin{equation}\label{eq:CN-constr-viol-5}
|u^m|^2 -1 =\mu \tau^2 |d_t u^m|^2 + (1-\mu)\tau^2 |d_t u^1|^2 + \mu \tau^4 \sum_{n=2}^m  |d_t^2 u^n|^2   + (1-2\mu)\tau^2\sum_{n=2}^m |d_tu^n|^2
\end{equation}
and%
\begin{equation}\label{eq:CN-constr-viol-a}
\begin{split}
\| |u^m|^2 -1\|_{L^1} & \leqslant \tau^2 \Big(\mu\|d_t u^m\|^2 +(1-\mu) \|d_t u^1\|^2 \\  
& \qquad\quad + \mu\tau^2 \sum_{n=2}^m  \|d_t^2 u^n\|^2 +|1-2\mu| \sum_{n=2}^m\|d_tu^n\|^2     \Big).
\end{split}
\end{equation}
\end{proposition}

\begin{proof}
Summing in \eqref{eq:orthog-time-deriv2} over $n$ from $n=2$ to $n=m$ and multiplying the result by $\tau$,
we immediately obtain 
\begin{equation*}
|u^m|^2 -|u^1|^2 =\mu\tau^2\Big ( |d_t u^m|^2-|d_t u^1|^2+ \tau^2 \sum_{n=2}^m  |d_t^2 u^n|^2\Big )
+\left(1- 2\mu\right) \tau^2\sum_{n=2}^m |d_tu^n|^2.
\end{equation*}
As $|u^0|^2 =1$ and $u^1$ is computed by the linearly implicit Euler method,
we have $|u^1|^2 =|u^0|^2+ \tau^2|d_t u^1|^2=1+ \tau^2|d_t u^1|^2$,
which yields \eqref{eq:CN-constr-viol-5}.
Integration of \eqref{eq:CN-constr-viol-5} over $\varOmega$ then leads to the constraint violation relation \eqref{eq:CN-constr-viol-a}. 
\end{proof}

From \cref{prp:Contraint violation}, we immediately deduce the following properties:\\
(a) If $0 \leqslant \mu\leqslant 1/2$, then $|u^m|\geqslant 1$ almost everywhere in $\varOmega$.\\
(b) If $\mu=0$, \eqref{eq:CN-constr-viol-5} reduces to the first-order constraint violation \eqref{eq:iE-constr-viol} of the linearly implicit Euler method.\\
(c) If $\mu = 1/2$ (midpoint and modified Euler methods), then \eqref{eq:CN-constr-viol-5} reduces to
\begin{equation*}
|u^m|^2 -1 = \frac{\tau^2}2 \Big( |d_t u^m|^2 + |d_t u^1|^2 + \tau^2 \sum_{n=2}^m  |d_t^2 u^n|^2 \Big).
\end{equation*}
In the following proposition,
we provide uniform upper bounds of the constraint violation error.

\begin{proposition}[bound of the constraint violation error]\label{prp:Contraint violation2}
Let $1/2 \leqslant \theta \leqslant 1$.\\
(a)
There exists $c_1>0$ such that $\| |u^m|^2 -1\|_{L^1} \leqslant c_1\tau$ for all $m\geqslant 2$ unconditionally.\\
(b)
Let $\mu=1/2$.
There exists $c_2>0$ such that $\| |u^m|^2 -1\|_{L^1} \leqslant c_2\tau^2$ for $m\geqslant 2$
if and only if there exists $c>0$ such that  the discrete regularity property 
\begin{equation}\label{eq:mild_reg-CN}
\|d_t u^m\|^2  + \|d_t u^1\|^2  + \tau^2 \sum_{n=2}^m  \|d_t^2u^n\|^2 \leqslant c
\end{equation}
is valid.
The constants $c_1,c_2>0$ depend on the energy of $u^0$
and the constant $c_\star>0$ in~\eqref{eq:metric_equivalence};
$c_2$ depends also on the constant $c$ in~\eqref{eq:mild_reg-CN}.
\end{proposition}

\begin{proof}
For all $n \geqslant 2$, the identity $\tau d_t^2 u^n=d_t u^n-d_t u^{n-1}$ implies 
\[ \tau^2 \|d_t^2 u^n\|^2\leqslant 2\big (\|d_t u^n\|^2+\|d_t u^{n-1}\|^2\big ),\]
and summation over $n$ yields  the inverse inequality
\begin{equation}\label{eq:inv-ineq}
\tau^2 \sum_{n=2}^m  \|d_t^2 u^n\|^2
\leqslant 4 \sum_{n=1}^m  \|d_t u^n\|^2.
\end{equation}
Combining the inverse inequality \eqref{eq:inv-ineq} with the energy stability~\eqref{eq:CV-energ-stab}
from \cref{pro:energy}, we see that the bound in~\eqref{eq:CN-constr-viol-a} is of order $\tau$ unconditionally.
If $\mu=1/2$, by design, the bound is of order $\tau^2$ if and only if the sharp discrete regularity condition \eqref{eq:mild_reg-CN} holds.
\end{proof}

The unconditional stability and the control of the constraint violation
allow to apply weak compactness arguments and show the weak convergence
of a subsequence of approximations toward a harmonic map; see, e.g., the proofs in~\cite[Ch.~7]{Bartels-book}.

\begin{remark}[beyond harmonic maps]
The proposed approach is general and can be applied to a vast class of geometrically constrained partial differential equations.
As an example, consider the LLG equation (see, e.g., \cite{BKP,Al2,AKST,DFPPRS,AFKL})
\begin{equation*}
\partial_t m = - m \times \varDelta m + \alpha \, m \times \partial_t m,
\end{equation*}
which models the dynamics of the magnetization $m$, a unit-length vector field, in ferromagnetic materials.
Here, $\alpha>0$ denotes the so-called Gilbert damping constant.
Applying the approach to this problem leads to the following method:
Given $m^0$ and $m^1$ (with $m^1$ computed by one step of the linearly implicit Euler method),
for $n \geqslant 2$, set ${\widehat m}^{n-1+\mu}:= m^{n-1}+\mu\tau\d_tm^{n-1}$
and compute $d_t m^n \in H^1_\DD(\varOmega;\R^\ell)$ such that $d_t m^n \cdot {\widehat m}^{n-1+\mu} = 0$ and
\begin{equation*}
\alpha (d_t m^n, v)
+ ({\widehat m}^{n-1+\mu} \times d_t m^n, v)
+ (\nabla [m^{n-1}+\theta\tau d_t m^n],\nabla v) = 0
\end{equation*}
for all $v\in H^1_\DD(\varOmega;\R^\ell)$ with $v\cdot {\widehat m}^{n-1+\mu} = 0$.
Combining the analysis of this work with those in~\cite{Al2,AHPPRS},
one can show that the method is unconditionally stable, formally of second order in time
if $\theta = \mu = 1/2$,
and convergent toward a weak solution of the problem.
\end{remark}

\subsection{The linearly implicit two-step BDF method}\label{SSe:2.2}
In this subsection, we recall the linearly implicit two-step BDF method for harmonic maps proposed in~\cite{ABP}, 
along with its energy decay and constraint violation properties.
This will provide the theoretical foundations for the experimental comparison of the algorithms
carried out in \cref{Se:num}.

The discrete time derivative $\dot{u}^n$ associated with the two-step backward difference formula (BDF2) 
is expressed as 
\[\dot{u}^n = \frac{1}{2\tau}\big(3 u^n - 4 u^{n-1} + u^{n-2}\big).\]
Following \cite{AFKL}, let 
${\widehat u}^n := u^{n-1} + \tau d_t u^{n-1} = 2 u^{n-1} - u^{n-2}$
be the extrapolated value at $t_n=n\tau$ of the linear interpolant based on $(t_{n-2},u^{n-2})$
and $(t_{n-1},u^{n-1}).$ Notice that ${\widehat u}^n={\widehat u}^{n-1+\mu}$ for $\mu=1$ 
is an approximation to $u(t_n).$

For the given initial value $u^0,$ one can compute the second starting
approximation $u^1$ by a single step method, for instance, by employing
one step with the linearly implicit Euler method \eqref{eq:iEa}.
Then, for given approximations  $u^{n-2}$ and $u^{n-1},$ 
we first seek $\dot{u}^n\in H_\DD^1(\varOmega; \R^\ell)$, 
satisfying the linearized unit-length condition $\dot{u}^n \cdot {\widehat u}^n = 0,$
such that
\begin{equation}\label{eq:BDF2a}
(\dot{u}^n, v)_\star + \frac13 (\nabla [4 u^{n-1} - u^{n-2} + 2 \tau \dot{u}^n],\nabla v) = 0
\end{equation}
for all $v\in H_\DD^1(\varOmega; \R^\ell)$ with $v\cdot {\widehat u}^n=0$. 
Notice that  the new approximation
\[u^n:=\frac13 (4 u^{n-1} - u^{n-2} + 2 \tau \dot{u}^n)\]
satisfies the required Dirichlet boundary condition $u^n|_{\varGamma_\DD}=u_\DD,$
provided the approximations $u^{n-2}$ and $u^{n-1}$ satisfy this condition. 

The algorithm of the BDF2 method~\cite{ABP} is summarized as follows. 

\begin{algorithm}[BDF2 method]
\label{alg:bdf2_iter}
\begin{algorithmic}
\STATE{Choose $u^0\in H^1(\varOmega;\R^\ell)$ with $u^0|_{\varGamma_\DD} = u_\DD$ and $|u^0| = 1$.}
\STATE{(0) Compute $d_t u^1 \in H^1_\DD(\varOmega;\R^\ell)$ such that $d_t u^1 \cdot u^0 = 0$ and
\[
(d_t u^1, v)_\star + (\nabla [u^0+\tau d_t u^1],\nabla v) = 0
\]
for all $v\in H^1_\DD(\varOmega;\R^\ell)$ with $v\cdot u^0 = 0$; set $u^1 = u^0+\tau d_t u^1$ and $n=2$.}
\STATE{(1) Set $\hu^n = 2 u^{n-1} - u^{n-2}$ and 
compute $\dot{u}^n \in H^1_\DD(\varOmega;\R^\ell)$ with $\dot{u}^n \cdot \hu^n = 0$ and
\[
(\dot{u}^n, v)_\star + \frac13 (\nabla [4 u^{n-1} - u^{n-2} + 2 \tau \dot{u}^n],\nabla v) = 0
\]
for all $v\in H^1_\DD(\varOmega;\R^\ell)$ with $v\cdot \hu^n = 0$; set 
$u^n =(4 u^{n-1} - u^{n-2} + 2 \tau \dot{u}^n)/3$.}
\STATE{(2) Stop if $\|\dot{u}^n\|_\star + 2\tau \|\nabla \dot{u}^n\|/3 \leqslant \veps_{\text{stop}}$ or if $n \tau \geqslant T$.}
\STATE{(3) Increase $n \to n+1$ and continue with~(1).}
\end{algorithmic}
\end{algorithm}

\begin{remark}
(i) Notice also that \eqref{eq:BDF2a} can be written in the form
\begin{equation*}
(\dot{u}^n, v)_\star + (\nabla u^n,\nabla v) = 0
 \end{equation*}
for all $v\in H_\DD^1(\varOmega; \R^\ell)$ with $v\cdot {\widehat u}^n=0$.\\
(ii) Similarly to \cref{alg}, the first stopping criterion used in step~(2) of \cref{alg:bdf2_iter} 
is chosen in such a way that the following relationship holds
\begin{equation}\label{stop 2}
(\nabla u_h^{N_\stop-1},\nabla v) \approx R(v)
\text{ for all admissible } v
\quad
\text{and}
\quad
\| R \|_{H^1(\varOmega)^\star} \leqslant \veps_{\text{stop}} . 
\end{equation}
Note that the first condition in \eqref{stop 1} holds with equality sign for the implicit Euler, modified Euler, and midpoint methods, 
while the first condition in \eqref{stop 2} holds approximately for the BDF2 method.
The aim is to ensure a fair comparison of all four methods in numerical experiments.
\end{remark}

We now recall the main properties of the BDF2 method established in~\cite{ABP}:\\
(i) \emph{Energy decay property}:
Let $G \in \R^{2 \times 2}$ be the positive definite, symmetric matrix given by
\[G:=\frac 14 \begin{pmatrix*}[r] 5 & -2\\ -2 & 1
\end{pmatrix*}\!. \]
Let ${\mathcal U}^n=(u^n,u^{n-1})^\top$ and let  $\|\cdot\|_G$ be a BDF-adapted variant of the $L^2$ norm,
\[\|{\mathcal U}^n\|_G^2=(G{\mathcal U}^n, {\mathcal U}^n)=\frac 54 \|u^n\|^2 - (u^n, u^{n-1}) + \frac 14 \|u^{n-1}\|^2.\]
Then, the BDF2 method satisfies the following energy identity (G-stability) for all $m \geqslant 2$
\begin{equation*}
\|\nabla {\mathcal U}^m\|_G^2  + \tau  \sum_{n=2}^m \|\dot{u}^n\|_\star^2 
+  \frac{\tau^4}{4} \sum_{n=2}^m \|d_t^2 \nabla u^n\|^2
=  \|\nabla {\mathcal U}^1\|_G^2.
\end{equation*}
\emph{Constraint violation result}:  If $u^1$ is computed by the linearly implicit Euler method \eqref{eq:iEa},
the sequence $(|u^n|)_{n \geqslant 1}$ is increasing almost everywhere in $\varOmega$, and
we have the following constraint violation result:
For all $m\geqslant 2$, it holds that
\begin{equation}\label{eq:BDF2-constr-viol-a}
\big \|  |u^m|^2 - 1 \big \|_{L^1} =\frac 32 \Big (1-\frac 1{3^m}\Big )\tau^2 \|d_tu^1\|^2
+\frac 32\tau^4\sum_{n=2}^m\Big (1-\frac 1{3^{m+1-n}}\Big ) \|d_t^2 u^n\|^2.
\end{equation}
It follows that
$\big\||u^m|^2- 1\big\|_{L^1} \leqslant c_p \tau^p$
unconditionally for $p=1$.
The result holds for $p=2$ under the sharp discrete regularity condition
\begin{equation}\label{eq:mild_reg}
\|d_t u^1\|^2  + \tau^2 \sum_{n=2}^m  \|d_t^2u^n\|^2 \leqslant c.
\end{equation}

\section{Variable step sizes}\label{Se:3}

In this section, we extend
the $(\theta,\mu)$-method to allow for variable step sizes,
aiming to improve the accuracy of the numerical solutions near singularities
(e.g., a singularity at $t = 0$ for nonsmooth initial data or a finite time blow-up even for smooth initial data)
and to apply acceleration techniques
to speed up the convergence to stationary states,
retaining the unconditional energy stability and the constraint violation properties. 
The analysis of the BDF2 method depends crucially on its G-stability property; cf.\ \cite{ABP};
consequently, the extension to variable step sizes is cumbersome.

In this section, we set $t_0:=0,$ and, for given positive step sizes $(\tau_n)_{n \geqslant 1}$,
we define the nodes $t_n:=t_{n-1}+\tau_n$, $n \geqslant 1$.
Using this notation, we denote by  $$d_t u^n := \frac{u^n-u^{n-1}}{\tau_n}$$
the backward difference quotient.

For given starting approximation $u^0$,
we compute $u^1$ by one step of the implicit Euler method:
First, we determine $d_t u^1\in H_\DD^1(\varOmega; \R^\ell)$, 
satisfying the linearized unit-length condition $d_t u^1 \cdot u^{0} = 0,$
such that
\begin{equation} \label{eq:ie_var_tau}
(d_t u^1,v)_\star + (\nabla u^0+\tau_1 \nabla d_t u^1,\nabla v) = 0
\end{equation}
for all $v\in H_\DD^1(\varOmega; \R^\ell)$ with $v\cdot u^0=0$. 
Then, we define $u^1:=u^0+\tau_1  d_t u^1$.

Next, for $n \geqslant 2$, let $s_n:=\tau_n/\tau_{n-1}$ denote the ratio of two consecutive step sizes.
For given approximations  $u^{n-2}$ and $u^{n-1}$, and $0\leqslant \mu\leqslant 1$,
let
\[{\widehat u}^{n-1+\mu} := (1+\mu s_n ) u^{n-1} - \mu s_n u^{n-2}=u^{n-1} +\mu\tau_n d_tu^{n-1}\]
be the extrapolated value at $t_{n-1+\mu}=t_{n-1}+\mu\tau_n$ 
of the linear interpolant based on $(t_{n-2},u^{n-2})$ and $(t_{n-1},u^{n-1})$.
Then, seek $d_t u^n\in H_\DD^1(\varOmega; \R^\ell)$, 
satisfying the linearized unit-length condition $d_t u^n \cdot \widehat u^{n-1+\mu} = 0,$
such that
\begin{equation}\label{eq:CNa-vart}
(d_t u^n, v)_\star +  (\nabla \big [u^{n-1} + \theta  \tau_n d_t u^n\big ],\nabla v) = 0
\end{equation}
for all $v\in H_\DD^1(\varOmega; \R^\ell)$ with $v\cdot \widehat u^{n-1+\mu}=0$. 
Here, $0 < \theta \leqslant 1$.
Thus, having computed $d_t u^n,$ we define the new approximation
$u^n:=u^{n-1} +  \tau_n d_t u^n$
that satisfies the required Dirichlet boundary condition $u^n|_{\varGamma_\DD}=u_\DD$,
provided the approximation $u^{n-1}$ at the previous step satisfies this condition. 

\begin{remark}
Similarly as the $(\theta,\mu)$-method \eqref{eq:CNa} with constant step size for harmonic maps,
\eqref{eq:CNa-vart} reduces to
\begin{enumerate}[(i)]\itemsep=0pt
\item  the linearly implicit Euler method for $\theta=1$ and $\mu=0$;
\item the linearly implicit midpoint method for $\theta=1/2$ and $\mu=1/2$;
\item a modified linearly implicit Euler method for $\theta=1$ and $\mu=1/2$.
\end{enumerate}
\end{remark}

We now extend the analysis of the previous section to the case of variable step sizes.
We discuss the energy decay and constraint violation properties of the method.
It can be proved that the iteration \eqref{eq:CNa-vart} becomes stationary for $n\rightarrow\infty$.

We begin with the discrete energy law satisfied by the approximations.

\begin{proposition}[energy decay]\label{pro:energy2}
For the $(\theta,\mu)$-method \eqref{eq:CNa-vart} and $m\geqslant 2$,
if $u^1$ is computed by the linearly implicit Euler method~\eqref{eq:ie_var_tau},
we have 
\begin{equation}\label{eq:CV-energ-stab-variable}
\frac{1}{2}\|\nabla u^m\|^2
+ \sum_{n=1}^m \tau _n\|d_t u^n\|_\star^2 
+\frac{\tau_1^2}2 \|\nabla d_tu^1\|^2
+\left(\theta-\frac12\right)\sum_{n=2}^m \tau_n^2 \|\nabla d_tu^n\|^2 
=\frac12 \|\nabla u^0\|^2.
 \end{equation}
\end{proposition}

The proof is analogous to the one of \cref{pro:energy}
and is therefore omitted.

We now study the constraint violation properties.
In the case of variable step sizes, the second difference quotient $d_t^2 u^n$
is defined by
\begin{equation}\label{eq:CN-diff-quot-vart}
d_t^2 u^n := \frac{1}{\tau_n} (d_t u^n - d_tu^{n-1}),\quad n\geqslant 2.
\end{equation}
We have the following analogue of \cref{Le:orthog-time-deriv}.

\begin{lemma}[discrete chain rule]\label{Le:orthog-time-deriv-variable}
For a sequence $(v^n)_{n \geqslant 0}$ and $n \geqslant 2$ we have
\[2 d_t v^n \cdot {\widehat v}^{n-1+\mu}=d_t|v^n|^2- \mu\tau_n^2 d_t|d_t v^n|^2- \mu \tau_n^3 |d_t^2 v^n|^2
- \left(1-2\mu \right)\tau_n|d_tv^n|^2  .\]
\end{lemma}

We omit the proof as it follows along the lines of that of \cref{Le:orthog-time-deriv}.
\Cref{Le:orthog-time-deriv-variable} implies that,
if $d_t u^n \cdot {\widehat u}^{n-1+\mu} = 0$,
for $n\geqslant 2$ we have
\begin{equation}\label{eq:orthog-time-deriv2-variable}
d_t|u^n|^2=\mu \tau_n^2 d_t|d_t u^n|^2+\mu \tau_n^3 |d_t^2 u^n|^2 
+(1-2\mu)\tau_n|d_tu^n|^2 .
\end{equation}

This relation is the crucial ingredient for the identity established in the following proposition.

\begin{proposition}[constraint violation error]
For the $(\theta,\mu)$-method \eqref{eq:CNa-vart} and $m\geqslant 2$,
if $|u^0| =1$ and $u^1$ is computed by the linearly implicit Euler method~\eqref{eq:ie_var_tau},
we have 
\begin{equation}\label{eq:CN-constr-viol-7-vart}
\begin{aligned}
|u^m|^2 -1 
& =
\mu \tau_m^2|d_t u^m|^2 + (1-\mu)\tau_1^2|d_t u^1|^2 
+\mu\sum_{n=1}^{m-1} \tau_n^2(1-s_{n+1}^2)|d_t u^n|^2   \\
&\quad+\mu \sum_{n=2}^m \tau_n^4 |d_t^2 u^n|^2
+(1-2\mu)\sum_{n=2}^m\tau_n^2|d_tu^n|^2.
\end{aligned}
\end{equation}
Thus, we have the constraint violation estimate
\begin{equation}\label{eq:CN-constr-viol-a-vart}
\begin{aligned}
 \| |u^m|^2 -1\|_{L^1} 
& \leqslant \mu\tau_m^2\|d_t u^m\|^2+ (1-\mu)\tau_1^2\|d_t u^1\|^2
+\mu\sum_{n=1}^{m-1} \tau_n^2|1-s_{n+1}^2|\|d_t u^n\|^2 \\
&\quad + \mu\sum_{n=2}^m \tau_n^4 \|d_t^2 u^n\|^2 
+ |1-2\mu|\sum_{n=2}^m \tau_n^2\|d_tu^n\|^2 .
\end{aligned}
\end{equation}
\end{proposition}

\begin{proof}
We multiply \eqref{eq:orthog-time-deriv2-variable} by $\tau_n$ and write it in the form
\begin{equation*}
|u^n|^2- |u^{n-1}|^2=\mu \tau_n^2\big (|d_t u^n|^2-|d_t u^{n-1}|^2\big )+\mu\tau_n^4 |d_t^2 u^n|^2 +\left(1-2\mu\right)\tau_n^2 |d_tu^n|^2 .
\end{equation*}
Summing over $n$ from $n=2$ to $n=m$, we obtain
\begin{equation*}
\begin{split}
|u^m|^2-|u^1|^2
& = \mu \sum_{n=2}^{m} \tau_n^2\big (|d_t u^n|^2-|d_t u^{n-1}|^2\big ) \\
& \qquad +\mu \sum_{n=2}^{m} \tau_n^4 |d_t^2 u^n|^2
+(1-2\mu)\sum_{n=2}^{m} \tau_n^2 |d_tu^n|^2.
\end{split}
\end{equation*}
Shifting the indices in the sum and using the definition of the step size ratio ($s_{n+1} = \tau_{n+1}/\tau_n$),
the first term on the right-hand side can be rewritten as
\begin{equation*}
\sum_{n=2}^m \tau_n^2\big (|d_t u^n|^2-|d_t u^{n-1}|^2\big )
=\tau_m^2|d_t u^m|^2-\tau_2^2|d_t u^1|^2
+\sum_{n=2}^{m-1} \tau_n^2(1-s_{n+1}^2)|d_t u^n|^2.
\end{equation*}
We thus obtain
\begin{equation*}
\begin{split}
|u^m|^2&-|u^1|^2\\ 
&=\mu \Big (\tau_m^2|d_t u^m|^2-\tau_2^2|d_t u^1|^2
+\sum_{n=2}^{m-1} \tau_n^2(1-s_{n+1}^2)|d_t u^n|^2
+ \sum_{n=2}^m \tau_n^4 |d_t^2 u^n|^2\Big )\\
&\quad 
+(1-2\mu)\sum_{n=2}^m \tau_n^2|d_tu^n|^2 .
\end{split}
\end{equation*}
If $|u^0| =1$ and $u^1$ is computed by the implicit Euler method,
then $|u^1|^2 =|u^0|^2+ \tau_1^2|d_t u^1|^2=1+ \tau_1^2|d_t u^1|^2,$
which yields \eqref{eq:CN-constr-viol-7-vart}.
Taking the $L^1$ norm, we obtain the constraint violation estimate
\eqref{eq:CN-constr-viol-a-vart}. 
\end{proof}

\begin{remark}
(i)
Notice that in the case of constant step size, we have $s_n=1$ for all $n \geqslant 2$,
and the bound in the constraint violation estimate \eqref{eq:CN-constr-viol-a-vart} 
reduces to the expression on the right-hand side of~\eqref{eq:CN-constr-viol-a}.\\
(ii)
If $0 \leqslant \mu \leqslant 1/2$ and if the step size monotonically decreases,
i.e., if $s_n\leqslant 1$ for all $2 \leqslant n \leqslant m$,
it follows from \eqref{eq:CN-constr-viol-7-vart}
that $|u^m|\geqslant 1$.\\
(iii)
From~\eqref{eq:CN-constr-viol-7-vart},
for $\mu=0$,
we obtain the constraint violation identity
\begin{equation*}
|u^m|^2 -1 
= \sum_{n=1}^m\tau_n^2|d_tu^n|^2
\end{equation*}
of the implicit Euler method.
For $\mu=1/2$,
we obtain the constraint violation identity
\begin{equation*}
|u^m|^2 -1 
=
\frac{1}{2} \left(
\tau_1^2|d_t u^1|^2 
+ \tau_m^2|d_t u^m|^2
+ \sum_{n=1}^{m-1} \tau_n^2(1-s_{n+1}^2)|d_t u^n|^2
+ \sum_{n=2}^m \tau_n^4 |d_t^2 u^n|^2 \right)
\end{equation*}
of the midpoint and modified Euler methods.
\end{remark}

\section{Numerical experiments}\label{Se:num}
 
We illustrate the accuracy in approximating the unit-length constraint and the overall performance 
of the proposed $(\theta,\mu)$-method (\cref{alg})
through a series of numerical experiments in two dimensions,
comparing it with the BDF2 method (\cref{alg:bdf2_iter}).
Moreover, for the $(\theta,\mu)$-method, we present numerical results obtained using variable step sizes.
 
In all computations,
the domain of the problem is $\varOmega = (-1/2,1/2)^2$ with
$\varGamma_\DD = \partial\varOmega$,
and the vector fields attain values in $\R^3$.
For the spatial discretization,
we consider a fixed unstructured triangular mesh $\mathcal{T}_h$ of $\varOmega$ generated by Netgen~\cite{netgen}
(consisting of 4901 vertices and 9544 triangles, and having mesh size 
$h \approx \num{2.286529e-02}$).
For the approximation of vector-valued functions,
we use $H^1$-conforming first-order finite elements,
i.e.,
we consider the space $V_h \subset H^1(\varOmega; \R^3)$ of 
vector-valued $\mathcal{T}_h$-piecewise affine and globally continuous functions.
The pointwise orthogonality imposed in the variational formulation of the methods
is enforced only at the vertices of $\mathcal{T}_h$.
For all methods, the solution of the arising constrained linear system is based on the null-space method given
in~\cite{RG,KPPRS}.
Moreover, we consider
a fixed tolerance $\veps_{\text{stop}} = 10^{-6}$
or a fixed final time $T=1$
in the stopping criterion of step~(2).

\subsection{Constant step size}

We start with a collection of numerical experiments to assess the performance of the schemes
in the case of constant step size.

\subsubsection{Comparison of \cref{alg,alg:bdf2_iter} ($H^1$-gradient flow)} \label{sec:stereo}

We consider the model problem from~\cite[Section~4.1]{ABP}.
Let $u_\DD = \pi_{\mathrm{st}}^{-1} \vert_{\varGamma_\DD}$,
where
\begin{equation} \label{eq:invstereo}
\pi_{\mathrm{st}}^{-1}(x) = (|x|^2+1)^{-1} (2x, 1 - |x|^2)^{\top}
\end{equation}
denotes the inverse stereographic projection.
It is well known that $u = \pi_{\mathrm{st}}^{-1}$ is a harmonic map with $u\vert_{\varGamma_\DD} = u_\DD$.

The discrete Dirichlet data are obtained interpolating the exact solution on the boundary.
To initialize the iterative algorithms,
we consider a fixed initial guess $u_h^0 = \mathcal{I}_hu^0$,
where $\mathcal{I}_h: C(\overline{\varOmega}; \R^3) \to V_h$
denotes the nodal interpolation operator (or its scalar-valued counterpart) and
\begin{equation} \label{eq:invstereo_perturbed}
u^0(x) = \left([u_1(x) + \varphi(x)]^2 + [u_2(x) - \varphi(x)]^2+u_3(x)^2\right)^{-1/2}
\begin{pmatrix} u_1(x) + \varphi(x) \\ u_2(x) - \varphi(x) \\ u_3(x) \end{pmatrix}
\end{equation}
with
\begin{equation*}
\varphi(x) = 16 \sin(4 \pi x_1) (x_1^2 - 1/4) (x_2^2 - 1/4).
\end{equation*}
 
In our first numerical experiment,
we compare the proposed $(\theta,\mu)$-method
and the
BDF2 method for the case of the $H^1$-gradient flow,
i.e., $(\cdot,\cdot)_\star = (\nabla\cdot,\nabla\cdot)$,
considering the step sizes $\tau = 2^{-m}$ for $m=4,\dotsc,10$.
For the $(\theta,\mu)$-method,
we consider the relevant cases of the
implicit Euler method ($\theta=1$, $\mu=0$)
and the midpoint method ($\theta=1/2$, $\mu=1/2$).
The results are displayed in \cref{tab:invstereo_H1}.
We compare the algorithms and test their convergence with respect to $\tau$
by assessing the following quantities:
\begin{enumerate}[\small $\bullet$]\itemsep=0pt
\item $N_\stop$, the number of iterations required to meet the tolerance;
\item $\delta_\infty[u_h^{N_\stop}] := \| |u_h^{N_\stop}| - 1\|_{L^\infty}$, the constraint violation error,
measured using the $L^\infty$-norm
(not unconditionally controlled by the algorithms);
\item $\delta_\uni[u_h^{N_\stop}] := \| \mathcal{I}_h|u_h^{N_\stop}|^2 -1 \|_{L^1}$,
the constraint violation error, measured using the $L^1$-norm
(this is the quantity that should decay linearly for the implicit Euler method
and, under a sharp discrete regularity condition, quadratically for the BDF2 method
and the midpoint method);
\item $A^2$, $B^2$, and $C^2$,
the quantities that, if uniformly bounded with respect to $\tau$,
guarantee the second-order convergence of $\delta_\uni[u_h^{N_\stop}]$ as $\tau \to 0$
for the BDF2 method (see~\eqref{eq:mild_reg}) and the midpoint method (see~\eqref{eq:mild_reg-CN}).
Specifically, we have
\[
A^2 = \tau^2 \sum_{n=2}^{N_\stop} \|d_t^2 u_h^n\|^2,
\quad B^2 = \|d_t u_h^1\|^2, \quad \text{and} \quad C^2 = \|d_t u_h^{N_\stop}\|^2.
\]
Note that $C^2$ is relevant only for the midpoint method,
and that none of these quantities plays a role in the analysis
of the implicit Euler method;
\item $\delta_\ener[u_h^{N_\stop}] = | I[u_h^{N_\stop}] - I[u] |$,
the energy approximation error.
\end{enumerate}
Finally,
$\mathrm{eoc}_\infty$ and $\mathrm{eoc}_\uni$
denote the experimental rates $q$ of the convergences of
$\delta_\infty[u_h^{N_\stop}]$ and $\delta_\uni[u_h^{N_\stop}]$
as $\tau \to 0$, respectively,
computed as logarithmic slopes $$q = -\log(\mathrm{err}_{k+1}/\mathrm{err}_k)/\log2$$
for any two consecutive instances of the errors.

\begin{table}[ht]
\begin{center}
\sisetup{output-exponent-marker=\ensuremath{\mathrm{e}}}
\tiny\setlength{\tabcolsep}{3pt}
\begin{tabular}{ l r c c c c c }
\hline
\multicolumn{7}{c}{Implicit Euler method}\\ \hline
$\tau$ & $N_\stop$
& $\delta_\infty[u_h^{N_\stop}]$ & $\mathrm{eoc}_\infty$
& $\delta_\uni[u_h^{N_\stop}]$ & $\mathrm{eoc}_\uni$
& $\delta_\ener[u_h^{N_\stop}]$\\ \hline
$2^{-4}$ & 274 & \num{1.067814e-02} & --- & \num{4.789695e-03} & --- & \num{1.757081e-02} \\
$2^{-5}$ & 535 & \num{5.445057e-03} & \num{0.971641312697510} & \num{2.435925e-03} & \num{0.975464075799131} & \num{8.419817e-03} \\
$2^{-6}$ & 1057 & \num{2.749793e-03} & \num{0.985624132496249} & \num{1.228481e-03} & \num{0.987594169825975} & \num{4.007746e-03} \\
$2^{-7}$ & 2101 & \num{1.381812e-03} & \num{0.992761673217135} & \num{6.169022e-04} & \num{0.993761848446522} & \num{1.846426e-03} \\
$2^{-8}$ & 4190 & \num{6.926474e-04} & \num{0.996368321640701} & \num{3.091205e-04} & \num{0.996872458388359} & \num{7.774544e-04} \\
$2^{-9}$ & 8368 & \num{3.467607e-04} & \num{0.998180718491386} & \num{1.547281e-04} & \num{0.998434105995996} & \num{2.459567e-04} \\
$2^{-10}$ & 16723 & \num{1.734898e-04} & \num{0.999089555688400} & \num{7.740610e-05} & \num{0.999216059066143} & \num{1.903686e-05} \\
\hline
\end{tabular}

\bigskip

\begin{tabular}{ l r c c c c c c c}
\hline
\multicolumn{9}{c}{BDF2 method}\\ \hline
$\tau$ & $N_\stop$
& $\delta_\infty[u_h^{N_\stop}]$ & $\mathrm{eoc}_\infty$
& $\delta_\uni[u_h^{N_\stop}]$ & $\mathrm{eoc}_\uni$
& $\delta_\ener[u_h^{N_\stop}]$
& $A^2$ & $B^2$\\ \hline
$2^{-4}$ & 262 & \num{1.449420e-03} & --- & \num{6.914510e-04} & --- & \num{2.106153e-03} & \num{3.771055e-03} & \num{1.134456e-01} \\
$2^{-5}$ & 523 & \num{3.777684e-04} & \num{1.939901775084087} & \num{1.805268e-04} & \num{1.937413990280181} & \num{3.341723e-04} & \num{1.988483e-03} & \num{1.204253e-01} \\
$2^{-6}$ & 1046 & \num{9.649027e-05} & \num{1.969046651942153} & \num{4.615075e-05} & \num{1.967787028914623} & \num{1.260331e-04} & \num{1.018802e-03} & \num{1.241592e-01} \\
$2^{-7}$ & 2090 & \num{2.438743e-05} & \num{1.984245737698221} & \num{1.166939e-05} & \num{1.983624945149232} & \num{2.437293e-04} & \num{5.153281e-04} & \num{1.260916e-01} \\
$2^{-8}$ & 4179 & \num{6.130565e-06} & \num{1.992045786218998} & \num{2.934096e-06} & \num{1.991741168182128} & \num{2.735198e-04} & \num{2.591147e-04} & \num{1.270748e-01} \\
$2^{-9}$ & 8357 & \num{1.536893e-06} & \num{1.996003313816719} & \num{7.356363e-07} & \num{1.995851498526484} & \num{2.810156e-04} & \num{1.299159e-04} & \num{1.275707e-01} \\
$2^{-10}$ & 16712 & \num{3.847574e-07} & \num{1.997995747539211} & \num{1.841743e-07} & \num{1.997920912030096} & \num{2.828958e-04} & \num{6.504700e-05} & \num{1.278197e-01} \\
\hline
\end{tabular}

\bigskip

\begin{tabular}{ l r c c c c c c c c}
\hline
\multicolumn{10}{c}{midpoint method}\\ \hline
$\tau$ & $N_\stop$
& $\delta_\infty[u_h^{N_\stop}]$ & $\mathrm{eoc}_\infty$
& $\delta_\uni[u_h^{N_\stop}]$ & $\mathrm{eoc}_\uni$
& $\delta_\ener[u_h^{N_\stop}]$
& $A^2$ & $B^2$ & $C^2$\\ \hline
$2^{-4}$ & 263 & \num{4.840786e-04} & --- & \num{2.307523e-04} & --- & \num{5.068923e-04} & \num{3.907769e-03} & \num{1.134456e-01} & \num{7.337834e-15} \\
$2^{-5}$ & 524 & \num{1.259806e-04} & \num{1.942039729996901} & \num{6.019202e-05} & \num{1.938700889407611} & \num{7.805042e-05} & \num{2.021934e-03} & \num{1.204253e-01} & \num{7.966500e-15} \\
$2^{-6}$ & 1046 & \num{3.216698e-05} & \num{1.969549187487585} & \num{1.538459e-05} & \num{1.968086237040558} & \num{2.310549e-04} & \num{1.026998e-03} & \num{1.241592e-01} & \num{8.306622e-15} \\
$2^{-7}$ & 2090 & \num{8.129365e-06} & \num{1.984365924597209} & \num{3.889857e-06} & \num{1.983696971753057} & \num{2.702607e-04} & \num{5.173513e-04} & \num{1.260916e-01} & \num{8.484000e-15} \\
$2^{-8}$ & 4179 & \num{2.043535e-06} & \num{1.992075712202860} & \num{9.780359e-07} & \num{1.991757792289502} & \num{2.801892e-04} & \num{2.596170e-04} & \num{1.270748e-01} & \num{8.513727e-15} \\
$2^{-9}$ & 8357 & \num{5.122987e-07} & \num{1.996009816920948} & \num{2.452123e-07} & \num{1.995856074714085} & \num{2.826877e-04} & \num{1.300409e-04} & \num{1.275707e-01} & \num{8.528770e-15} \\
$2^{-10}$ & 16712 & \num{1.282525e-07} & \num{1.997998282571460} & \num{6.139145e-08} & \num{1.997921697055550} & \num{2.833144e-04} & \num{6.507821e-05} & \num{1.278197e-01} & \num{8.551566e-15} \\
\hline
\end{tabular}
\end{center}
\footnotesize
\caption{Comparison of the implicit Euler, BDF2 and midpoint methods ($H^1$-gradient flow).}
\label{tab:invstereo_H1}
\end{table}

Looking at \cref{tab:invstereo_H1},
we see that $N_\stop \propto \tau^{-1}$ for all algorithms.
For any fixed $\tau$, the algorithms require approximately the same number of iterations to satisfy the stopping criterion.
In particular, for the BDF2 and the midpoint methods, the number of iterations is almost identical.
For the implicit Euler method, we see a clear first-order convergence of the $L^1$-error (as expected from the theory).
We also see a first-order convergence of the $L^\infty$-error.
For the BDF2 method and the midpoint method,
we observe a second-order convergence of the constraint violation error measured in both the $L^1$-norm and the $L^\infty$-norm.
The midpoint method is slightly more accurate, in the sense that for each value of $\tau$
the constraint violation error is about $1/3$ of the one of the BDF2 method.
The quantities $A^2$, $B^2$ (and $C^2$ for the midpoint method) stay uniformly bounded
as $\tau \to 0$, which is a numerical validation of the discrete regularity property.
For both methods,
the quantity $A^2$ converges linearly to $0$ as $\tau \to 0$.
The energy approximation error decays for the implicit Euler method,
whereas it stabilizes at about $2.8 \cdot 10^{-4}$
for both the BDF method and the midpoint method.
We believe that the lack of convergence of the energy error as $\tau \to 0$
for the higher-order methods
is due to the spatial approximation error becoming predominant.

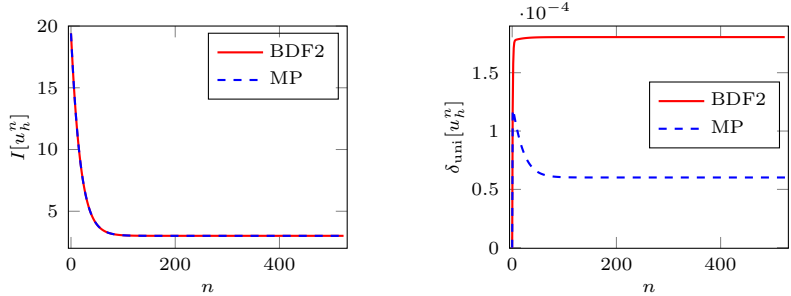
\begin{figure}[htbp]
\centering
\begin{subfigure}{0.35\textwidth}
\centering
\begin{tikzpicture}
\pgfplotstableread{pics/stereo_bdf2_energies.dat}{\energyBDF}
\pgfplotstableread{pics/stereo_cn_energies.dat}{\energyCN}
\begin{axis}
[
width = \textwidth,
xlabel = {\tiny $n$},
ylabel = {\tiny $I[u_h^n]$},
xmin = -5,
xmax = 530,
ymin = 2,
ymax = 20,
legend pos= north east,
legend cell align= left,
legend style = {fill=none},
]
\addplot[red, thick] table[x=iter, y=energy]{\energyBDF};
\addplot[blue, dashed, thick] table[x=iter, y=energy]{\energyCN};
\legend{
BDF2,
MP,
}

\end{axis}
\end{tikzpicture}
\end{subfigure}
\quad
\begin{subfigure}{0.35\textwidth}
\centering
\begin{tikzpicture}
\pgfplotstableread{pics/stereo_bdf2_L1error.dat}{\errorBDF}
\pgfplotstableread{pics/stereo_cn_L1error.dat}{\errorCN}
\begin{axis}
[
width = \textwidth,
xlabel = {\tiny $n$},
ylabel = {\tiny $\delta_\uni[u_h^{n}]$},
xmin = -5,
xmax = 530,
ymin = -0.5e-6,
ymax = 1.9e-4,
legend cell align= left,
legend style = {at={(0.97,0.6)},anchor=east,fill=none},
]
\addplot[red, thick] table[x=iter, y=L1error]{\errorBDF};
\addplot[blue, dashed, thick] table[x=iter, y=L1error]{\errorCN};
\legend{
BDF2,
MP,
}
\end{axis}
\end{tikzpicture}
\end{subfigure}
\caption{Evolution of the energy (left)
and the constraint violation $L^1$-error (right)
for the approximations generated
by the BDF2 and the midpoint (MP) methods
($H^1$-gradient flow, $\tau = 2^{-5}$).
}
\label{fig:plots_stereo}
\end{figure}

In \cref{fig:plots_stereo}, for both the BDF2 method and the midpoint method,
we show the evolution of the energy
and the constraint violation $L^1$-error during the iteration
for $\tau = 2^{-5}$.
On the one hand, we see that the energetic behavior of the methods is very similar,
as the curves are nearly superimposed
(for $\tau = 2^{-5}$ the difference between the energy values obtained with the two schemes is of the order of 0.01\%).
For both methods, the energy monotonically decays in an exponential fashion
and converges to the approximate energy value of the inverse stereographic projection.
On the other hand, the behavior of the schemes is clearly different as far as the evolution of the constraint violation error is concerned.
While the error is monotonically increasing for the BDF2 method (as expected from the theory, see~\eqref{eq:BDF2-constr-viol-a}),
for the midpoint method it reaches its maximal value in the very first iteration,
and then decays monotonically until it stabilizes to its final value.
For the sake of readability, in the plots we have omitted the results for the implicit Euler method.
However, we report that the energetic behavior is very similar,
whereas the constraint violation error, in agreement with~\eqref{eq:iE-constr-viol},
has the same monotonically increasing behavior of the BDF2 method
(but the value is one order of magnitude larger).

\subsubsection{Comparison of \cref{alg,alg:bdf2_iter} ($L^2$-gradient flow)} \label{sec:stereo2}

We repeat the experiment of \cref{sec:stereo}
for the implicit Euler, midpoint, and BDF2 methods, this time focusing on the $L^2$-gradient flow
(i.e., $(\cdot,\cdot)_\star = (\cdot,\cdot)$) instead of the $H^1$-gradient flow. 
The spatial discretization and the stopping tolerance are the same as before, while
we consider the values $\tau = 2^{-m}$ for $m=10,\dotsc,16$ for the step size.

\begin{table}[ht]
\begin{center}
\sisetup{output-exponent-marker=\ensuremath{\mathrm{e}}}
\tiny\setlength{\tabcolsep}{3pt}
\begin{tabular}{ l r c c c c c }
\hline
\multicolumn{7}{c}{Implicit Euler method}\\ \hline
$\tau$ & $N_\stop$
& $\delta_\infty[u_h^{N_\stop}]$ & $\mathrm{eoc}_\infty$
& $\delta_\uni[u_h^{N_\stop}]$ & $\mathrm{eoc}_\uni$
& $\delta_\ener[u_h^{N_\stop}]$\\ \hline
$2^{-10}$ & 358 & \num{2.717896e-02} & --- & \num{1.418881e-02} & --- & \num{6.181589e-02} \\
$2^{-11}$ & 704 & \num{1.429713e-02} & \num{0.926764682531987} & \num{7.534582e-03} & \num{0.913154215216359} & \num{2.920025e-02} \\
$2^{-12}$ & 1394 & \num{7.338287e-03} & \num{0.962210335531921} & \num{3.894810e-03} & \num{0.951974525682967} & \num{1.395834e-02} \\
$2^{-13}$ & 2776 & \num{3.718566e-03} & \num{0.980696950967261} & \num{1.982461e-03} & \num{0.974260466438524} & \num{6.696962e-03} \\
$2^{-14}$ & 5538 & \num{1.871943e-03} & \num{0.990209872407960} & \num{1.000518e-03} & \num{0.986545362220900} & \num{3.169622e-03} \\
$2^{-15}$ & 11063 & \num{9.391799e-04} & \num{0.995063068202221} & \num{5.026608e-04} & \num{0.993090032576087} & \num{1.433537e-03} \\
$2^{-16}$ & 22113 & \num{4.703980e-04} & \num{0.997519606469530} & \num{2.519426e-04} & \num{0.996490102430804} & \num{5.726016e-04} \\
\hline
\end{tabular}

\bigskip

\begin{tabular}{ l r c c c c c c c}
\hline
\multicolumn{9}{c}{BDF2 method}\\ \hline
$\tau$ & $N_\stop$
& $\delta_\infty[u_h^{N_\stop}]$ & $\mathrm{eoc}_\infty$
& $\delta_\uni[u_h^{N_\stop}]$ & $\mathrm{eoc}_\uni$
& $\delta_\ener[u_h^{N_\stop}]$
& $A^2$ & $B^2$\\ \hline
$2^{-10}$ & 347 & \num{1.609006e-02} & --- & \num{7.326965e-03} & --- & \num{3.100461e-02} & \num{6.151266e+02} & \num{4.441482e+03} \\
$2^{-11}$ & 691 & \num{5.911905e-03} & \num{1.444474714169849} & \num{2.318522e-03} & \num{1.660010307891790} & \num{8.392529e-03} & \num{5.268518e+02} & \num{5.857205e+03} \\
$2^{-12}$ & 1382 & \num{2.111226e-03} & \num{1.485542063765896} & \num{6.787015e-04} & \num{1.772356309482390} & \num{2.096575e-03} & \num{3.977257e+02} & \num{7.058192e+03} \\
$2^{-13}$ & 2763 & \num{7.092779e-04} & \num{1.573658121583390} & \num{1.875916e-04} & \num{1.855181973510716} & \num{3.567499e-04} & \num{2.692145e+02} & \num{7.954286e+03} \\
$2^{-14}$ & 5526 & \num{2.237489e-04} & \num{1.664470406371208} & \num{4.982835e-05} & \num{1.912556521137424} & \num{1.150726e-04} & \num{1.669773e+02} & \num{8.551995e+03} \\
$2^{-15}$ & 11051 & \num{6.861735e-05} & \num{1.705235274535649} & \num{1.290337e-05} & \num{1.949218895344054} & \num{2.400249e-04} & \num{9.725875e+01} & \num{8.917550e+03} \\
$2^{-16}$ & 22101 & \num{2.067605e-05} & \num{1.730612813924130} & \num{3.291053e-06} & \num{1.971126741217712} & \num{2.724355e-04} & \num{5.438053e+01} & \num{9.128111e+03} \\
\hline
\end{tabular}

\bigskip

\begin{tabular}{ l r c c c c c c c c}
\hline
\multicolumn{10}{c}{midpoint method}\\ \hline
$\tau$ & $N_\stop$
& $\delta_\infty[u_h^{N_\stop}]$ & $\mathrm{eoc}_\infty$
& $\delta_\uni[u_h^{N_\stop}]$ & $\mathrm{eoc}_\uni$
& $\delta_\ener[u_h^{N_\stop}]$
& $A^2$ & $B^2$ & $C^2$\\ \hline
$2^{-10}$ & 347 & \num{5.513829e-03} & --- & \num{2.490751e-03} & --- & \num{8.887947e-03} & \num{7.141146e+02} & \num{4.441482e+03} & \num{9.240650e-13} \\
$2^{-11}$ & 692 & \num{2.013096e-03} & \num{1.453638553072969} & \num{7.797516e-04} & \num{1.675494289750504} & \num{2.438945e-03} & \num{5.826272e+02} & \num{5.857205e+03} & \num{9.678072e-13} \\
$2^{-12}$ & 1382 & \num{7.107345e-04} & \num{1.502033336314724} & \num{2.271200e-04} & \num{1.779559853744712} & \num{4.909651e-04} & \num{4.258449e+02} & \num{7.058192e+03} & \num{9.953261e-13} \\
$2^{-13}$ & 2764 & \num{2.371833e-04} & \num{1.583308297574130} & \num{6.263214e-05} & \num{1.858479675617902} & \num{7.186503e-05} & \num{2.817768e+02} & \num{7.954286e+03} & \num{9.896444e-13} \\
$2^{-14}$ & 5526 & \num{7.468650e-05} & \num{1.667083037220113} & \num{1.662029e-05} & \num{1.913957618176593} & \num{2.275157e-04} & \num{1.720716e+02} & \num{8.551995e+03} & \num{9.973991e-13} \\
$2^{-15}$ & 11051 & \num{2.291029e-05} & \num{1.704851770942925} & \num{4.302301e-06} & \num{1.949765187222118} & \num{2.690342e-04} & \num{9.923837e+01} & \num{8.917550e+03} & \num{9.987112e-13} \\
$2^{-16}$ & 22101 & \num{6.904625e-06} & \num{1.730360753584621} & \num{1.097158e-06} & \num{1.971337161984462} & \num{2.798282e-04} & \num{5.516288e+01} & \num{9.128111e+03} & \num{9.993687e-13} \\
\hline
\end{tabular}
\end{center}
\footnotesize
\caption{Comparison of the implicit Euler, BDF2 and midpoint methods ($L^2$-gradient flow).}
\label{tab:invstereo_L2}
\end{table}

The results of the simulations are displayed in \cref{tab:invstereo_L2}.
For the implicit Euler method,
the expected linear decay of the constraint violation error
(measured both with respect to the $L^1$-norm and the $L^\infty$-norm) is clearly visible.
However,
differently from what we observed for the $H^1$-gradient flow,
this behavior emerges only for the smallest step sizes
(there is a much longer preasymptotic phase than for the $H^1$-gradient flow).
The energy approximation error converges to $0$ linearly as $\tau \to 0$.
In terms of accuracy, the performance of the $H^1$-gradient flow is significantly better,
e.g., the smallest constraint violation error obtained with the $L^2$-gradient flow,
$\delta_\uni[u_h^{N_\stop}] = \num{2.519426e-04}$,
requires $\tau = 2^{-16}$ and 22113 iterations.
With the $H^1$-gradient flow,
the choice $\tau = 2^{-9}$ results only in 8368 iterations and leads to a smaller error
($\delta_\uni[u_h^{N_\stop}] = \num{1.547281e-04}$).

A long preasymptotic phase is observed also for the BDF2 method with $L^2$-gradient flow
(much longer than for the version with $H^1$-gradient flow).
Indeed, the quadratic convergence guaranteed by the method
is (almost) seen only for the smallest step sizes.
On the other hand, as we saw for the case of the $H^1$-gradient flow,
the convergence of the energy approximation is spoiled by the spatial approximation error.
A difference between the two considered gradient flow metrics
is visible also looking at the order of magnitude of $A^2$ and $B^2$,
with the values for the $L^2$-gradient flow being significantly larger than those for the $H^1$-gradient flow.
The quantity $A^2$ tends to $0$ as $\tau$ increases,
whereas we can observe a slight growth of $B^2$.

Similarly, in the case of the midpoint method,
we see the expected second-order convergence of the $L^1$-error
only when the step size becomes sufficiently small.
Again, the energy approximation error does not converge.
The quantity $A^2$ tends to $0$ as $\tau$ decreases,
whereas we can observe a slight growth of both $B^2$ and $C^2$.

\subsubsection{Modified implicit Euler method} \label{sec:modified}

We repeat the experiment for the modified implicit Euler method (\cref{alg} with $\theta=1$ and $\mu=1/2$),
considering both the $H^1$-gradient flow
and the $L^2$-gradient flow. 
The modified implicit Euler method is a mixture of the standard Euler and midpoint methods, for which the variational 
formulation to be solved at each iteration is the one of the standard Euler method. 
However, like in the midpoint method,
for $n \geqslant 2$, the orthogonality constraint is considered with respect to extrapolated value ${\widehat u}_h^{n-1/2}$, instead of $u_h^{n-1}$.

\begin{table}[ht]
\begin{center}
\sisetup{output-exponent-marker=\ensuremath{\mathrm{e}}}
\tiny\setlength{\tabcolsep}{3pt}
\begin{tabular}{ l r c c c c c c c c}
\hline
\multicolumn{10}{c}{Modified implicit Euler method ($H^1$-gradient flow)}\\ \hline
$\tau$ & $N_\stop$
& $\delta_\infty[u_h^{N_\stop}]$ & $\mathrm{eoc}_\infty$
& $\delta_\uni[u_h^{N_\stop}]$ & $\mathrm{eoc}_\uni$
& $\delta_\ener[u_h^{N_\stop}]$
& $A^2$ & $B^2$ & $C^2$\\ \hline
$2^{-4}$ & 271 & \num{4.829655e-04} & --- & \num{2.303562e-04} & --- & \num{5.055126e-04} & \num{3.706005e-03} & \num{1.134456e-01} & \num{7.036515e-15} \\
$2^{-5}$ & 532 & \num{1.259041e-04} & \num{1.939594871316770} & \num{6.016462e-05} & \num{1.936879168422335} & \num{7.814522e-05} & \num{1.966096e-03} & \num{1.204253e-01} & \num{7.795351e-15} \\
$2^{-6}$ & 1054 & \num{3.216196e-05} & \num{1.968898030023679} & \num{1.538279e-05} & \num{1.967598164010376} & \num{2.310611e-04} & \num{1.012302e-03} & \num{1.241592e-01} & \num{8.215564e-15} \\
$2^{-7}$ & 2099 & \num{8.129044e-06} & \num{1.984197727164654} & \num{3.889742e-06} & \num{1.983570818845384} & \num{2.702611e-04} & \num{5.135808e-04} & \num{1.260916e-01} & \num{8.318066e-15} \\
$2^{-8}$ & 4187 & \num{2.043515e-06} & \num{1.992032863801813} & \num{9.780286e-07} & \num{1.991725907951079} & \num{2.801892e-04} & \num{2.586620e-04} & \num{1.270748e-01} & \num{8.490180e-15} \\
$2^{-9}$ & 8365 & \num{5.122974e-07} & \num{1.995999358211607} & \num{2.452119e-07} & \num{1.995847659869674} & \num{2.826877e-04} & \num{1.298006e-04} & \num{1.275707e-01} & \num{8.516978e-15} \\
$2^{-10}$ & 16721 & \num{1.282524e-07} & \num{1.997995746496659} & \num{6.139142e-08} & \num{1.997920048670634} & \num{2.833144e-04} & \num{6.501795e-05} & \num{1.278197e-01} & \num{8.530437e-15} \\
\hline
\end{tabular}

\bigskip

\begin{tabular}{ l r c c c c c c c c}
\hline
\multicolumn{10}{c}{Modified implicit Euler method ($L^2$-gradient flow)}\\ \hline
$\tau$ & $N_\stop$
& $\delta_\infty[u_h^{N_\stop}]$ & $\mathrm{eoc}_\infty$
& $\delta_\uni[u_h^{N_\stop}]$ & $\mathrm{eoc}_\uni$
& $\delta_\ener[u_h^{N_\stop}]$
& $A^2$ & $B^2$ & $C^2$\\ \hline
$2^{-10}$ & 347 & \num{5.461846e-03} & --- & \num{2.461069e-03} & --- & \num{8.764122e-03} & \num{6.532986e+02} & \num{4.441482e+03} & \num{9.185912e-13} \\
$2^{-11}$ & 700 & \num{1.995531e-03} & \num{1.452615946155298} & \num{7.757144e-04} & \num{1.665687618098933} & \num{2.423257e-03} & \num{5.499176e+02} & \num{5.857205e+03} & \num{9.657152e-13} \\
$2^{-12}$ & 1390 & \num{7.109472e-04} & \num{1.488958366743549} & \num{2.266415e-04} & \num{1.775113528244151} & \num{4.891714e-04} & \num{4.105570e+02} & \num{7.058192e+03} & \num{9.945471e-13} \\
$2^{-13}$ & 2772 & \num{2.380126e-04} & \num{1.578704469535225} & \num{6.258296e-05} & \num{1.856570254397675} & \num{7.204605e-05} & \num{2.755945e+02} & \num{7.954286e+03} & \num{9.894180e-13} \\
$2^{-14}$ & 5534 & \num{7.495484e-05} & \num{1.666944405177144} & \num{1.661586e-05} & \num{1.913208930004023} & \num{2.275318e-04} & \num{1.698855e+02} & \num{8.551995e+03} & \num{9.973162e-13} \\
$2^{-15}$ & 11059 & \num{2.296291e-05} & \num{1.706716158481226} & \num{4.301938e-06} & \num{1.949502327907603} & \num{2.690354e-04} & \num{9.854490e+01} & \num{8.917550e+03} & \num{9.986803e-13} \\
$2^{-16}$ & 22109 & \num{6.911846e-06} & \num{1.732162501836424} & \num{1.097130e-06} & \num{1.971252250425153} & \num{2.798283e-04} & \num{5.495828e+01} & \num{9.128111e+03} & \num{9.993561e-13} \\
\hline
\end{tabular}
\end{center}
\footnotesize
\caption{Performance of the modified implicit Euler method ($H^1$- and $L^2$-gradient flows).}
 \label{tab:invstereo_modified}
\end{table}

The results of the simulations are displayed in \cref{tab:invstereo_modified}.
Comparing the performance of this method for both gradient flow metrics
with the previous three approaches, 
we can see that the modified implicit Euler method
behaves like the midpoint method.

\subsubsection{Performance in the presence of singularities} \label{sec:singular}

It is well known that the heat flow of harmonic maps can develop singularities,
even if the initial value is smooth.
In this subsection,
we test the performance of the
projection-free linearly implicit midpoint method
in such a situation.

\begin{figure}[ht]
\centering
\begin{subfigure}{0.25\textwidth}
	\centering
	\includegraphics[width=\textwidth]{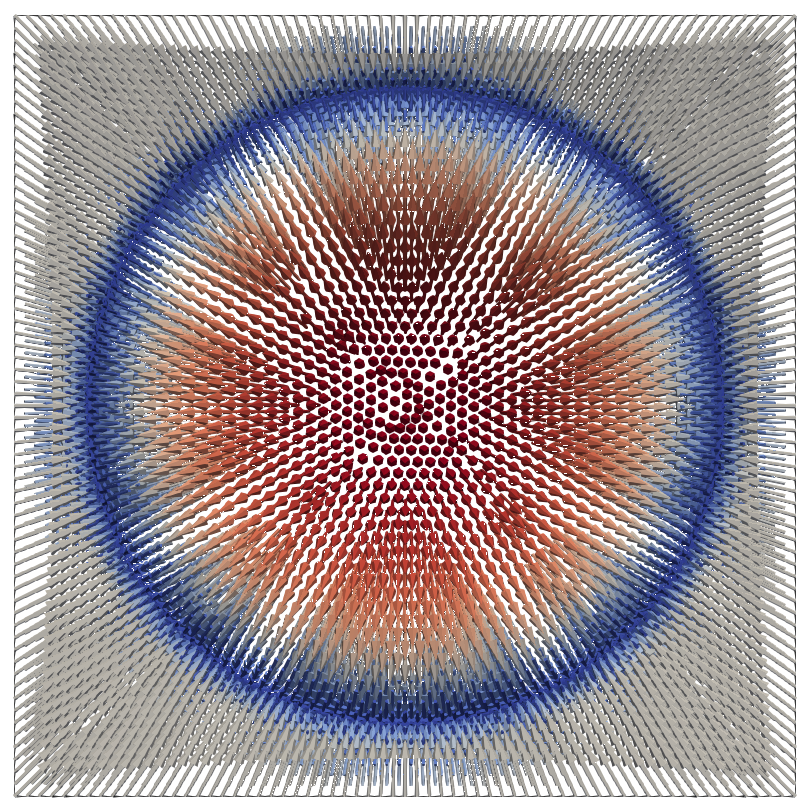}
	\caption{Initial value.}
\end{subfigure}
\quad
\begin{subfigure}{0.25\textwidth}
	\centering
	\includegraphics[width=\textwidth]{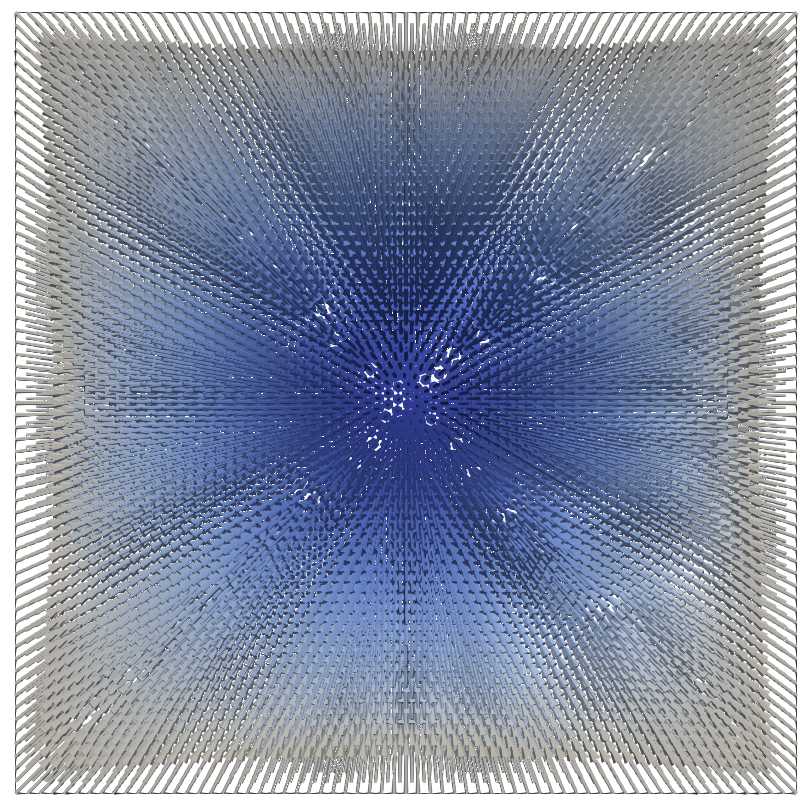}
	\caption{Final value.}
\end{subfigure}
\vspace*{-3mm}
\caption{Pictures of the initial value (left) and the final value (right).
The picture of the final value refers to the results
obtained for $\tau = 2^{-14}$.
The color scale refers to the third component of
the field, which attains values between -1 (blue) and 1 (red).}
\label{fig:exp_singular}
\end{figure}

The problem specifications are the same as in the previous experiments,
except for the initial value and the stopping criterion.
We consider the initial value~\cite{CDY}
\begin{equation} \label{eq:singular_ic}
u^0(x) = |x|^{-1} \big( x \sin\phi(2|x|) , |x| \cos\phi(2|x|) \big)^\top
\end{equation}
with
$\phi(s) = (3\pi/2) \min\{ s^2, 1\}$.
We refer to~\cite[Section~6.2]{Bartels-2016} for numerical results for the harmonic map heat flow
with this initial value
obtained with the linearly implicit Euler method~\eqref{eq:iEa}.
Restricting ourselves to
the $L^2$-gradient flow
and the midpoint method,
we simulate the gradient flow dynamics in the fixed time interval $[0,T]$ with $T=1$
using the values $\tau = 2^{-m}$ for $m=11,\dotsc,16$ for the step size.
For pictures of the initial value $u_h^0$ and the final value $u_h^N$ (with $N = 2^m$)
computed by the algorithm for the case $m=14$, we refer to \cref{fig:exp_singular}.

\begin{table}[ht]
\begin{center}
\sisetup{output-exponent-marker=\ensuremath{\mathrm{e}}}
\tiny\setlength{\tabcolsep}{3pt}
\begin{tabular}{ l r c c c c c c c}
\hline
\multicolumn{9}{c}{midpoint method}\\ \hline
$\tau$ & N
& $\delta_\infty[u_h^{N}]$ & $\mathrm{eoc}_\infty$
& $\delta_\uni[u_h^{N}]$ & $\mathrm{eoc}_\uni$
& $A^2$ & $B^2$ & $C^2$\\ \hline
$2^{-11}$ & 2048 & \num{6.695872e-01} & --- & \num{3.213439e-03} & --- & \num{1.446455e+04} & \num{7.644249e+03} & \num{1.924170e-10} \\
$2^{-12}$ & 4096 & \num{2.807488e-01} & \num{1.253992093073576} & \num{1.062308e-03} & \num{1.596915972062555} & \num{1.670666e+04} & \num{1.156007e+04} & \num{9.759367e-13} \\
$2^{-13}$ & 8192 & \num{6.973685e-02} & \num{2.009286751707688} & \num{3.036247e-04} & \num{1.806841049709376} & \num{1.479558e+04} & \num{1.683711e+04} & \num{8.836271e-13} \\
$2^{-14}$ & 16384 & \num{1.058308e-02} & \num{2.720161642383164} & \num{8.545439e-05} & \num{1.829062645716390} & \num{1.056962e+04} & \num{2.370806e+04} & \num{8.562739e-13} \\
$2^{-15}$ & 32768 & \num{1.310741e-03} & \num{3.013305011988832} & \num{2.693660e-05} & \num{1.665586846576346} & \num{7.706581e+03} & \num{3.220000e+04} & \num{8.497994e-13} \\
$2^{-16}$ & 65536 & \num{2.705663e-04} & \num{2.276328577054749} & \num{8.859118e-06} & \num{1.604332783414990} & \num{6.026424e+03} & \num{4.182021e+04} & \num{8.483792e-13} \\
\hline
\end{tabular}
\end{center}
\footnotesize
\caption{Performance of the
midpoint method for the harmonic map heat flow
in the presence of singularities.}
\label{tab:singular}
\end{table}

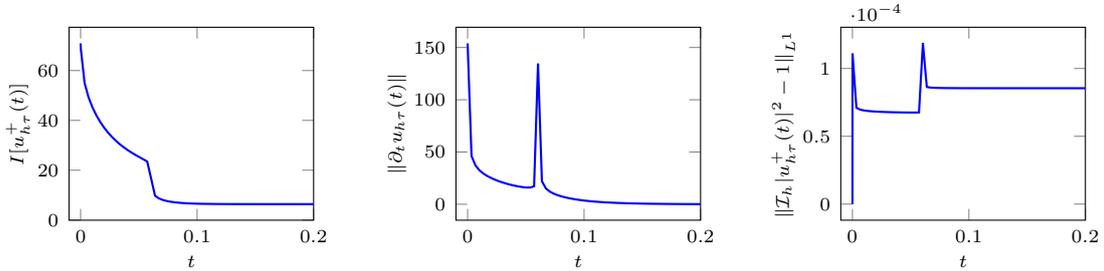
\begin{figure}[ht]
\begin{subfigure}{0.32\textwidth}
\centering
\begin{tikzpicture}
\pgfplotstableread{pics/hmhf_singular_energies.dat}{\energy}
\begin{axis}
[
width = \textwidth,
xlabel = {\tiny $t$},
ylabel = {\tiny $I[u_{h\tau}^+(t)]$},
xmin = -0.01,
xmax = 0.2,
xtick={0, 0.1, 0.2},
]
\addplot[blue, thick] table[x=t, y=energy]{\energy};
\end{axis}
\end{tikzpicture}
\end{subfigure}
\
\begin{subfigure}{0.32\textwidth}
\centering
\begin{tikzpicture}
\pgfplotstableread{pics/hmhf_singular_update_norms.dat}{\norms}
\begin{axis}
[
width = \textwidth,
xlabel = {\tiny $t$},
ylabel = {\tiny $\lVert \partial_t u_{h\tau}(t) \rVert$},
xmin = -0.01,
xmax = 0.2,
xtick={0, 0.1, 0.2},
]
\addplot[blue, thick] table[x=t, y=update_norm]{\norms};
\end{axis}
\end{tikzpicture}
\end{subfigure}
\
\begin{subfigure}{0.32\textwidth}
\centering
\begin{tikzpicture}
\pgfplotstableread{pics/hmhf_singular_L1error.dat}{\error}
\begin{axis}
[
width = \textwidth,
xlabel = {\tiny $t$},
ylabel = {\tiny $\lVert \mathcal{I}_h \vert u_{h\tau}^+(t) \vert^2 - 1 \rVert_{L^1}$},
xmin = -0.01,
xmax = 0.2,
xtick={0, 0.1, 0.2},
]
\addplot[blue, thick] table[x=t, y=L1error]{\error};
\end{axis}
\end{tikzpicture}
\end{subfigure}
\caption{Evolution in $[0,1/5]$ of the energy (left),
the $L^2$-norm of the update (middle),
and the constraint violation $L^1$-error (right)
for the approximations generated
by the midpoint method 
($\tau = 2^{-14}$).}
\label{fig:plots_singular}
\end{figure}

Looking at the results in \cref{tab:singular},
we observe the convergence of the constraint violation errors, but the rate is only superlinear.
In particular, the convergence of $\delta_\uni[u_h^{N}]$ is not of second order,
which is in agreement with the unboundedness of $B^2$,
for which we observe a moderate grow.
In \cref{fig:plots_singular}, for $\tau = 2^{-14}$,
we show the evolution in $[0,1/5]$ of the energy $I[u_{h\tau}^+(t)]$,
of the $L^2$-norm of $\lVert \partial_t u_{h\tau}(t) \rVert$,
and of the constraint violation error $\lVert \mathcal{I}_h \vert u_{h\tau}^+(t) \vert^2 - 1 \rVert_{L^1}$
(note that we omit the results for $t \in (1/5,1]$ because the curves are nearly constant,
as the evolution has already reached the stationary state).
Here, $u_{h\tau}(t)$ (resp., $u_{h\tau}^+(t)$)
denotes the globally continuous and piecewise affine interpolant
(resp., the forward piecewise constant interpolant)
of the sequence of snapshots $(u_h^n)_{n \geqslant 0}$.
Looking at the evolution of the energy,
we see an abrupt decay at $t \approx 0.06$,
which is when the singularity disappears.
In this phase, the dynamics is faster, as the spike in the plot of $\lVert \partial_t u_{h\tau}(t) \rVert$ reveals,
which leads to a local growth of $\lVert \mathcal{I}_h \vert u_{h\tau}^+(t) \vert^2 - 1 \rVert_{L^1}$.

\subsection{Variable step size}

In this section,
we present two numerical experiments to showcase the performance of the midpoint method
with variable step size.

\subsubsection{Prescribed variable step size} \label{sec:variable1}

Consider the projection-free linearly implicit midpoint method with variable step size $\tau_n$,
i.e., the iteration defined by~\eqref{eq:CNa-vart} with $\theta = \mu = 1/2$ for $n \geqslant 2$
(initialized by one step of the implicit Euler method with step size $\tau_1$ for $n=1$).
From the analysis of \cref{Se:3},
we obtain that the iterates generated by the scheme satisfy
the discrete energy law
\begin{equation} \label{eq:energy_cn_var}
\frac{1}{2}\|\nabla u^m\|^2+ \sum_{n=1}^m \tau _n\|d_t u^n\|_\star^2
+ \frac{\tau _1^2}2  \| \nabla d_t u^1\|^2
=\frac12 \|\nabla u^0\|^2
\end{equation}
and the upper bound of the constraint violation error
\begin{equation} \label{eq:con_viol_cn_var}
\begin{split}
 \| |u^m|^2 -1\|_{L^1} 
& \leqslant
\frac{1}{2} \left( \tau_1^2\|d_t u^1\|^2
+ \tau_m^2\|d_t u^m\|^2
+ \sum_{n=2}^m \tau_n^4 \|d_t^2 u^n\|^2
\right) \\
& \quad
+ \frac{1}{2} \sum_{n=1}^{m-1} \tau_n^2|1-s_{n+1}^2|\|d_t u^n\|^2;
\end{split}
\end{equation}
see~\eqref{eq:CV-energ-stab-variable} and~\eqref{eq:CN-constr-viol-a-vart},
respectively.
We now aim to exploit the flexibility given by the variable step size
to accelerate the convergence to stationary configurations.

First of all, motivated by the heuristic fact that the step size should increase as the iteration evolves toward convergence,
we aim to obtain a nondecreasing sequence of step sizes,
i.e., we impose that $\tau_n \leqslant \tau_{n+1}$ or, equivalently, $s_{n+1} \geqslant 1$.

Consider now the right-hand side of~\eqref{eq:con_viol_cn_var},
which consists of two terms.
Up to the generalization to variable step sizes,
the first term is the same term that is present
in the corresponding estimate for the method with constant step size.
This term decays quadratically with respect to the step size if a discrete regularity condition is satisfied.
The second term arises when the step size used in the method is not constant.
Motivated by the similarity of this term to the second term on the left-hand side of~\eqref{eq:energy_cn_var},
and observing that the latter is uniformly bounded,
we aim to define $\tau_{n+1}$ in such a way that 
the step size powers in the two terms on the right-hand side of~\eqref{eq:con_viol_cn_var} are balanced,
which is true if $|1-s_{n+1}^2| = s_{n+1}^2 -1 = c \tau_n$ for some $c>0$.
Manipulating this identity, we obtain the step size update
\begin{equation} \label{eq:step_size_adjustment}
\tau_{n+1} = \tau_n \sqrt{1 + c\tau_n}.
\end{equation}
To numerically validate this choice, we repeat the experiments of \cref{sec:stereo,sec:stereo2}
for the linearly implicit midpoint method with variable step size.
We consider the initial step size $\tau_1 = 2^{-m}$ for $m=6,\dotsc,10$
for the $H^1$-gradient flow
(resp., $m=10,\dotsc,16$ for the $L^2$-gradient flow)
and update it during the iteration using the aforementioned formula with $c=1$, i.e., $\tau_{n+1} = \tau_n \sqrt{1 + \tau_n}$.

\begin{table}[ht]
\begin{center}
\sisetup{output-exponent-marker=\ensuremath{\mathrm{e}}}
\tiny\setlength{\tabcolsep}{3pt}
\begin{tabular}{ l r c c c c c c c c}
\hline
\multicolumn{9}{c}{midpoint method with $\tau_{n+1} = \tau_n \sqrt{1 + \tau_n}$ ($H^1$-gradient flow)}\\ \hline
$\tau_1$ & $N_\stop$ & $\tau_{N_\stop}$ 
& $\delta_\infty[u_h^{N_\stop}]$
& $\delta_\uni[u_h^{N_\stop}]$
& $\delta_\ener[u_h^{N_\stop}]$
& $A^2$ & $B^2$ & $C^2$ \\ \hline
$2^{-6}$   & 139    & \num{183.4840}       & \num{1.774948e-05} & \num{5.604644e-06} & \num{3.009548e-04} & \num{1.434655e-03} & \num{1.241592e-01} & \num{1.247063e-18} \\
$2^{-7}$   & 266    & \num{9.4134}           & \num{4.709664e-06} & \num{1.487952e-06} & \num{2.882204e-04} & \num{7.249986e-04} & \num{1.260916e-01} & \num{2.318879e-16} \\
$2^{-8}$   & 522    & \num{4.0424}           & \num{1.215676e-06} & \num{3.842935e-07} & \num{2.847453e-04} & \num{3.644335e-04} & \num{1.270748e-01} & \num{3.591280e-16} \\
$2^{-9}$   & 1035  & \num{3.9375}           & \num{3.090855e-07} & \num{9.773934e-08} & \num{2.838355e-04} & \num{1.827027e-04} & \num{1.275707e-01} & \num{2.815904e-17} \\
$2^{-10}$ & 2059  & \num{2.1582}           & \num{7.795060e-08} & \num{2.465423e-08} & \num{2.836025e-04} & \num{9.147314e-05} & \num{1.278197e-01} & \num{1.969074e-16} \\
\hline
\end{tabular}

\bigskip

\begin{tabular}{ l r c c c c c c c c}
\hline
\multicolumn{9}{c}{midpoint method with $\tau_{n+1} = \tau_n \sqrt{1 + \tau_n}$ ($L^2$-gradient flow)}\\ \hline
$\tau_1$ & $N_\stop$ & $\tau_{N_\stop}$ 
& $\delta_\infty[u_h^{N_\stop}]$
& $\delta_\uni[u_h^{N_\stop}]$
& $\delta_\ener[u_h^{N_\stop}]$
& $A^2$ & $B^2$ & $C^2$ \\ \hline
$2^{-10}$ & 319     & \num{0.0012}           & \num{5.503357e-03} & \num{2.484544e-03} & \num{8.864545e-03} & \num{7.148302e+02} & \num{4.441482e+03} & \num{9.574041e-13} \\
$2^{-11}$ & 637     & \num{5.7799e-04}    & \num{2.010740e-03} & \num{7.780006e-04} & \num{2.432801e-03} & \num{5.830711e+02} & \num{5.857205e+03} & \num{9.616386e-13} \\
$2^{-12}$ & 1272   & \num{2.8897e-04}    & \num{7.101913e-04} & \num{2.266544e-04} & \num{4.893734e-04} & \num{4.260996e+02} & \num{7.058192e+03} & \num{9.925515e-13} \\
$2^{-13}$ & 2543   & \num{1.4449e-04}    & \num{2.370374e-04} & \num{6.251212e-05} & \num{7.227205e-05} & \num{2.819149e+02} & \num{7.954286e+03} & \num{9.970764e-13} \\
$2^{-14}$ & 5086   & \num{7.2246e-05}    & \num{7.464864e-05} & \num{1.658983e-05} & \num{2.276187e-04} & \num{1.721438e+02} & \num{8.551995e+03} & \num{9.933730e-13} \\
$2^{-15}$ & 10170 & \num{3.6122e-05}    & \num{2.290065e-05} & \num{4.294628e-06} & \num{2.690601e-04} & \num{9.927536e+01} & \num{8.917550e+03} & \num{9.977045e-13} \\
$2^{-16}$ & 20339 & \num{1.8061e-05}    & \num{6.902191e-06} & \num{1.095233e-06} & \num{2.798347e-04} & \num{5.518161e+01} & \num{9.128111e+03} & \num{9.983312e-13} \\
\hline
\end{tabular}
\end{center}
\footnotesize
\caption{Midpoint method with prescribed variable step size.}
\label{tab:invstereo_variable_step_size}
\end{table}

The results of the simulation are displayed in \cref{tab:invstereo_variable_step_size},
where we collect the same outputs considered before, plus the value $\tau_{N_\stop}$ of the step size at the final iteration.

Looking at the results obtained for the $H^1$-gradient flow and comparing them with those obtained for constant step size
(see the third table in \cref{tab:invstereo_H1}),
we see
that the quality of the results improves significantly
and the method with variable step size clearly overcomes the one using a constant step size.
Indeed,
for the same computational cost, i.e., if the number of iterations is approximately the same,
the constraint violation error observed for the method with variable step size is much smaller:
e.g.,
for about $2000$ iterations (obtained for $\tau= 2^{-7}$ in the case of constant step size),
the constraint violation error measured in the $L^1$-norm is of the order of $10^{-6}$;
for the same number of iterations, the method with variable step size yields an error of the order of $10^{-8}$.
Looking at the other quantities,
we see that the lack of convergence of the energy values is the same for both variants of the method
and that the discrete regularity condition seems to hold.

For the $L^2$-gradient flow
(compare the second table in \cref{tab:invstereo_variable_step_size}
with the third table in \cref{tab:invstereo_L2}),
we see that the difference in the performance
of the midpoint method in the case of constant and variable step sizes is negligible.
We conclude that the prescribed step size update rule is very effective for energy minimization in the case of the $H^1$-gradient flow,
but not very impactful for the $L^2$-gradient flow.

\subsubsection{Adaptive step size control} \label{sec:variable2}

In this subsection,
we endow the linearly implicit midpoint method
with a simple adaptive mechanism, which allows to adjust the step size
without assuming any knowledge on the specific problem data.

Specifically, for all $n \geqslant 2$,
after the computation of the update $d_t u_h^n$ via~\eqref{eq:CNa-vart}
and the definition of the new approximation $u_h^n = u_h^{n-1} + \tau_n d_t u_h^n$,
we adjust the step size $\tau_{n+1}$ for the next iteration according to the following rule:
Given $0 < \tau_{\min} < \tau_{\max}$,
if $\lVert d_t u_h^n \rVert > \lVert d_t u_h^{n-1} \rVert$, then we reduce the step size
for the next iteration as
$\tau_{n+1} = \max\big\{ \tau_{\min} , \tau_n \sqrt{1 - \tau_n/\tau_{\max}} \big\}$;
otherwise,
if $\lVert d_t u_h^n \rVert \leqslant \lVert d_t u_h^{n-1} \rVert$, then we enlarge the step size
as
$\tau_{n+1} = \min\big\{ \tau_{\max} , \tau_n \sqrt{1 + \tau_n/\tau_{\max}} \big\}$.
The approach reduces the step size whenever the gradient flow dynamics accelerates,
conversely the step size increases if the dynamics slows down.
The specific formula for the adjustment comes from the discussion in \cref{sec:variable1}
(cf.\ \eqref{eq:step_size_adjustment}),
where we choose $c = 1/\tau_{\max}$ to guarantee that the step size is always positive.

We test the approach using the setups of \cref{sec:singular}
(harmonic map heat flow with singular solutions).
We consider the initial step size $\tau_1 = 2^{-m}$ for $m=11,\dotsc,16$.
Moreover, we set $\tau_{\min} = 2^{-18}$ and $\tau_{\max} = 1$.

\begin{table}[ht]
\begin{center}
\sisetup{output-exponent-marker=\ensuremath{\mathrm{e}}}
\tiny\setlength{\tabcolsep}{3pt}
\begin{tabular}{ l r c c c c c c}
\hline
\multicolumn{7}{c}{adaptive midpoint method}\\ \hline
$\tau_0$ & $N$
& $\delta_\infty[u_h^{N}]$
& $\delta_\uni[u_h^{N}]$
& $A^2$ & $B^2$ & $C^2$ \\ \hline
$2^{-11}$ & 1721   & \num{7.052877e-01} & \num{3.477822e-03} & \num{1.434814e+04} & \num{7.644249e+03} & \num{7.488691e-05} \\
$2^{-12}$ & 3250   & \num{2.855156e-01} & \num{1.074735e-03} & \num{1.665691e+04} & \num{1.156007e+04} & \num{9.791702e-13} \\
$2^{-13}$ & 6497   & \num{7.343041e-02} & \num{3.103971e-04} & \num{1.494351e+04} & \num{1.683711e+04} & \num{8.837285e-13} \\
$2^{-14}$ & 12991 & \num{1.130145e-02} & \num{8.661054e-05} & \num{1.071492e+04} & \num{2.370806e+04} & \num{8.570663e-13} \\
$2^{-15}$ & 25981 & \num{1.402676e-03} & \num{2.707185e-05} & \num{7.785717e+03} & \num{3.220000e+04} & \num{8.495828e-13} \\
$2^{-16}$ & 51960 & \num{2.705590e-04} & \num{8.872591e-06} & \num{6.065570e+03} & \num{4.182021e+04} & \num{8.482394e-13} \\
\hline
\end{tabular}
\end{center}
\footnotesize
\caption{Performance of the midpoint method
with adaptively adjusted step size
for the harmonic map heat flow
in the presence of singularities.}
\label{tab:adaptive_singular}
\end{table}

\begin{figure}[ht]
\centering
\begin{subfigure}{0.32\textwidth}
\centering
\begin{tikzpicture}
\pgfplotstableread{pics/hmhf_singular_adaptive_energies.dat}{\energy}
\begin{axis}
[
width = \textwidth,
xlabel = {\tiny $t$},
ylabel = {\tiny $I[u_{h\tau}^+(t)]$},
xmin = -0.01,
xmax = 0.2,
xtick={0, 0.1, 0.2},
]
\addplot[blue, thick] table[x=t, y=energy]{\energy};
\end{axis}
\end{tikzpicture}
\end{subfigure}
\quad
\begin{subfigure}{0.32\textwidth}
\centering
\begin{tikzpicture}
\pgfplotstableread{pics/hmhf_singular_adaptive_update_norms.dat}{\norms}
\begin{axis}
[
width = \textwidth,
xlabel = {\tiny $t$},
ylabel = {\tiny $\lVert \partial_t u_{h\tau}(t) \rVert$},
xmin = -0.01,
xmax = 0.2,
xtick={0, 0.1, 0.2},
]
\addplot[blue, thick] table[x=t, y=update_norm]{\norms};
\end{axis}
\end{tikzpicture}
\end{subfigure}\\
\bigskip
\begin{subfigure}{0.32\textwidth}
\centering
\begin{tikzpicture}
\pgfplotstableread{pics/hmhf_singular_adaptive_L1error.dat}{\error}
\begin{axis}
[
width = \textwidth,
xlabel = {\tiny $t$},
ylabel = {\tiny $\lVert \mathcal{I}_h \vert u_{h\tau}^+(t) \vert^2 - 1 \rVert_{L^1}$},
xmin = -0.01,
xmax = 0.2,
xtick={0, 0.1, 0.2},
]
\addplot[blue, thick] table[x=t, y=L1error]{\error};
\end{axis}
\end{tikzpicture}
\end{subfigure}
\quad
\begin{subfigure}{0.32\textwidth}
\centering
\begin{tikzpicture}
\pgfplotstableread{pics/hmhf_singular_adaptive_tau.dat}{\stepsize}
\begin{axis}
[
width = \textwidth,
xlabel = {\tiny $t$},
ylabel = {\tiny $\tau$},
xmin = -0.01,
xmax = 0.2,
xtick={0, 0.1, 0.2},
]
\addplot[blue, thick] table[x=t, y=tau]{\stepsize};
\end{axis}
\end{tikzpicture}
\end{subfigure}
\caption{Evolution in $[0,1/5]$ of the energy (top-left),
the $L^2$-norm of the update (top-right),
the constraint violation $L^1$-error (bottom-left),
and the step size (bottom-right)
for the approximations generated
by the midpoint method with adaptively adjusted step size
($\tau_1 = 2^{-14}$).}
\label{fig:plots_singular_adaptive}
\end{figure}
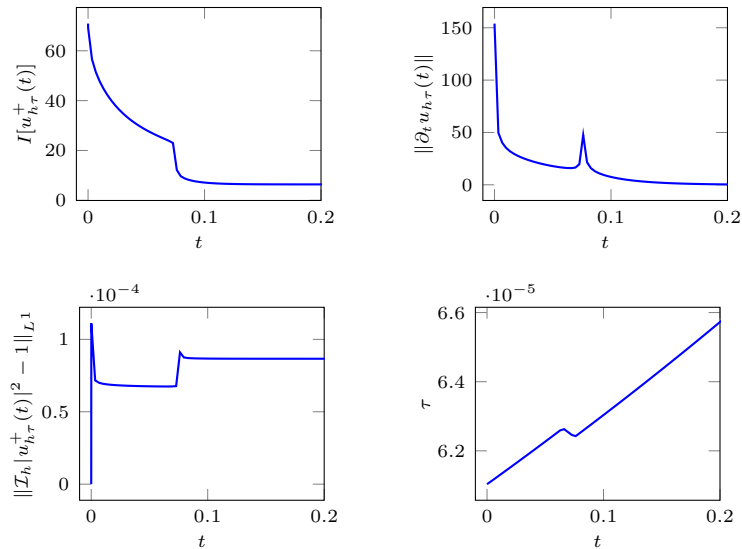

Comparing the results in \cref{tab:adaptive_singular} and \cref{fig:plots_singular_adaptive} with those obtained for constant step size (see \cref{tab:singular}),
we see that the number of iterations $N$ needed by the adaptive method to simulate the gradient flow dynamics in the interval of interest
(a measure of the overall computational cost of the simulation)
is reduced by 20\% despite achieving approximately the same accuracy in the realization of the unit-length constraint.
In \cref{fig:plots_singular_adaptive} (bottom-right),
we see that the reduction of the computational cost is obtained by increasing the step size throughout the entire evolution,
except for $t \in (0.06,0.08)$,
where the faster dynamics (cf.\ the plots in \cref{fig:plots_singular_adaptive} (top)) requires a more accurate time discretization.

\bibliographystyle{habbrv}
\bibliography{ref}

\end{document}